\documentclass[11pt,letterpaper,english]{smfart}
\usepackage[utf8]{inputenc}
\usepackage[english,french]{babel}
\usepackage{smfthm}
\usepackage{mathrsfs}
\usepackage{graphics}
\usepackage{hyperref} 
\usepackage{amssymb, amsmath, mathabx}
\usepackage{lmodern}
\usepackage[all]{xy}
\usepackage{multicol}
\usepackage{color}
\usepackage[shortlabels]{enumitem}
\usepackage{scalerel}
\usepackage{tikz-cd}
\makeindex
\DeclareUnicodeCharacter{00A0}{\relax}
\allowdisplaybreaks

\author{Antoine Ducros}
\address{Sorbonne Université, Université Paris-Diderot, CNRS, Institut de Mathématiques de Jussieu-Paris
Rive Gauche, Campus Pierre et Marie Curie, case 247, 4 place Jussieu, 75252 Paris cedex 5, France}
\email{antoine.ducros\at imj-prg.fr}
\urladdr{https://webusers.imj-prg.fr/$\sim$antoine.ducros/}
\author{Ehud Hrushovski}
\address{Mathematical Institute,
University of Oxford,
Andrew Wiles Building,
Radcliffe Observatory Quarter
Woodstock Road, 
Oxford
OX2 6GG, UK}
\email{Ehud.Hrushovski\at maths.ox.ac.uk}
\urladdr{https://www.maths.ox.ac.uk/people/ehud.hrushovski}
\author{François Loeser}
\address{Institut universitaire de France, Sorbonne Université, Institut de Mathématiques de Jussieu-Paris
Rive Gauche CNRS, Campus Pierre et Marie Curie, case 247, 4 place Jussieu, 75252 Paris cedex 5, France
}
\email{francois.loeser@imj-prg.fr}
\urladdr{https://webusers.imj-prg.fr/$\sim$francois.loeser/}

\author{Jinhe Ye}
\address{Mathematical Institute, University of Oxford, Andrew Wiles Building, Radcliffe Observatory Quarter Woodstock Road, Oxford OX2 6GG, UK}
\email{Jinhe.ye@maths.ox.ac.uk}
\urladdr{https://sites.google.com/view/vincentye}

\title{Tropical functions on a skeleton}
\NumberTheoremsAs{subsection}
\SwapTheoremNumbers

\newcommand{\eg}{e.\@g.\@}

\newcommand{\ie}{i.\@e.\@}

\newcommand{\opcit}{op.\@~cit.\@}
\newcommand{\resp}{resp.\@~}

\newcommand{\abs}[1]{\mathopen|#1\mathclose|}

\newcommand{\an}{^{\mathrm{an}}}

\newcommand{\gm}{\mathbf G_{\mathrm m}}

\newcommand{\cl}{\mathrm{cl}}

\newcommand{\val}{\mathrm{val}}

\newcommand{\tp}{\mathrm{tp\;}}
\newcommand{\acvsf}{\mathrm{ACV}^2\mathrm{F}}
\newcommand{\vsf}{\mathrm{V}^2\mathrm{F}}
\newcommand{\acvf}{\mathrm{ACVF}}
\newcommand{\DOAG}{\mathrm{DOAG}}
\newcommand{\RES}{\mathrm{RES}}
\newcommand{\res}{\mathrm{res}}
\newcommand{\VF}{\mathrm{VF}}

\newcommand{\F}{\mathbf F}

\renewcommand{\P}{\mathbf P}

\newcommand{\R}{\mathbf R}
\newcommand{\Z}{\mathbf Z}

\renewcommand{\SS}{\mathbb S}
\renewcommand{\H}{\mathrm H}

\renewcommand{\phi}{\varphi}
\renewcommand{\epsilon}{\varepsilon}
\renewcommand{\leq}{\leqslant}
\renewcommand{\geq}{\geqslant}
\renewcommand{\subset}{\subseteq}

\renewcommand{\labelitemi}{$\bullet$}

\renewcommand{\theequation}{\alph{equation}}
\SetEnumerateShortLabel{1}{\textnormal{(\arabic{enumi})}}
\SetEnumerateShortLabel{i}{\textnormal{(\roman{enumi})}}
\SetEnumerateShortLabel{a}{\textnormal{(\alph{enumi})}}
\SetEnumerateShortLabel{A}{\textnormal{(\Alph{enumi})}}
\SetEnumerateShortLabel{2}{\textnormal{(\arabic{enumii})}}
\SetEnumerateShortLabel{j}{\textnormal{(\roman{enumii})}}
\SetEnumerateShortLabel{b}{\textnormal{(\arabic{enumi}\alph{enumii})}}
\SetEnumerateShortLabel{c}{\textnormal{(\alph{enumii})}}
\SetEnumerateShortLabel{B}{\textnormal{(\Alph{enumii})}}

\theoremstyle{plain} \newtheorem{claim}[subsection]{Claim}

 \newtheorem*{mtheo}{Main Theorem}

\def\stda#1{{#1}^{\#}}

\begin{document}
\begin{abstract}
We prove a general finiteness statement for the ordered abelian group of tropical functions on skeleta in Berkovich analytifications of algebraic varieties. Our approach consists in working in the framework of stable completions of algebraic varieties, a model-theoretic version of Berkovich analy\-tifications, for which we prove a similar result, of which the former one is a 
consequence.
\end{abstract}

\maketitle

\tableofcontents

\section{Introduction}
\subsection{The general context: skeleta in Berkovich geometry}
Let $F$ be a complete non-archimedean field. Among the several
frameworks available
for doing analytic geometry over $F$ (Tate, Raynaud, Berkovich, Huber\ldots),
Berkovich's is the one that encapsulates in the most natural way the deep
links between non-archimedean and tropical (or polyhedral) geometry.

Indeed, every Berkovich space $X$ over $F$ contains plenty of natural ``tropical" subspaces, 
which are called \textit{skeleta}. Roughly speaking, a skeleton
of $X$ is a subset $S$ of $X$ on which the sheaf of functions of the form
$\log \abs f$ with $f$ a section of $\mathscr O_X^\times$ induces a piecewise linear structure; \ie, using such functions one can equip $S$ 
with a piecewise linear atlas, whose charts are modelled on
(rational) polyhedra and whose transition maps are piecewise affine (with rational linear part). 

This definition is rather abstract,
but there are plenty of concrete examples of skeleta.
The prototype of such objects is the ``standard skeleton"
$S_n$ of $(\gm^n)\an$,
that consists of all Gauss norms with arbitrary
real parameters; the family $(\log \abs{T_1},\ldots, \log
\abs{T_n})$ induces a piecewise-linear isomorphism $S_n\simeq \mathbb{R}^n$. 

Now if $X$ is an arbitrary analytic space
and if $\phi_1,\ldots, \phi_m$ are
quasi-finite maps from $X$ to $(\gm^n)\an$, then 
$\bigcup_j \phi_j^{-1}(S_n)$ is a skeleton by \cite{confluentes}, 
Theorem 5.1 (it consists only of
points whose Zariski-closure is $n$-dimensional, so
it is empty if $\dim X<n$), and $\phi_j^{-1}(S_n)
\to S_n$ is a piecewise immersion for all $j$; of
course, every piecewise-linear subspace
of $\bigcup_j \phi_j^{-1}(S_n)$ is still a skeleton. 

Skeleta were introduced by Berkovich in his 
seminal work \cite{berkovich1999} on the homotopy type of analytic spaces, where he proved
that any compact analytic space with a polystable formal model admits a deformation retraction
to a skeleton (isomorphic to the dual complex of the special fiber), and used it to show that quasi-smooth analytic spaces
are locally contractible; they play a key role
in the theory of real integration on Berkovich spaces \cite{chambertloir-d2012}. Let us mention that all skeleta encountered in these works are at least locally of the form described above; \ie, piecewise-linear subspaces of finite unions $\bigcup\phi_j^{-1}(S_n)$ for quasi-finite maps $\phi_j\colon X\to (\gm^n)\an$.

\subsection{Our main result}\label{sec:main-result}
If $S$ is a skeleton of an analytic space
$X$ and if $f$
is a regular invertible function
defined on a neighborhood of $S$, then $\log \abs f$
is a piecewise-linear function 
on $S$, and our purpose is to understand
what are the piecewise linear functions
on $S$ that can arise this way in the \textit{algebraic}
situation. 

Let us make precise what we mean. Let $X$ be an \textit{algebraic}
variety over $F$, say irreducible of
dimension $n$; let us call \textit{log-rational}
any real-valued function  of the form $\log \abs f$
for $f$ a non-zero rational function on $X$,
viewed as defined over $U\an$ for $U$ the maximal
open subset of $X$ on which $f$
is well-defined and invertible.
Let $\phi_1,\ldots, \phi_m$ be (algebraic)
quasi-finite maps from $X$ to $\gm^n$
(the corresponding analytic maps
will also be denoted $\phi_1,\ldots, \phi_m$). 
Let $S$ be a subset of the skeleton $\bigcup \phi_j^{-1}(S_n)$
defined by a Boolean combination of
inequalities between log-rational functions. 
Our main theorem is the following finiteness result. 

\begin{mtheo}[Berkovich setting]
 Let $X$ be an irreducible algebraic
variety over $F$ of
dimension $n$ and assume $F$ is algebraically closed. Let $S$ be as above. Then there exists finitely many non-zero rational functions
$f_1,\ldots, f_\ell$ on $X$ such that the following holds. 

\begin{enumerate}[1]
\item The functions $\log \abs{f_1},\ldots, \log \abs{f_\ell}$ identify $S$
with a piecewise-linear subset of $\mathbb{R}^\ell$ (\ie, a 
subset defined by a Boolean combination
of inequalities between $\mathbb{Q}$-affine functions). 
\item The group of restrictions of log-rational functions to $S$
is stable under $\min$ and $\max$
and is generated
under addition, substraction,
$\min$ and $\max$ by the (restrictions of the) functions $\log \abs{f_i}$ and the constants $\log \abs a$ for
$a\in F^\times$. 
\end{enumerate}
\end{mtheo}
Let us mention that statement (1) is
implicitly established in \cite{confluentes}
(see \opcit, proof of Theorem 5.1);
what is really new here is statement
(2). And let us insist on the assumption
that $F$ is algebraically closed:
for a general $F$ the theorem does
not hold, as shown by a
counter-example
due to Michael Temkin 
(Remark \ref{contrex}).

\subsection{About our proof}\label{sec:about-proof}
In fact, we do not work directly with Berkovich spaces 
but with the model-theoretic avatar of this geometry, namely the theory of  \textit{stable completions} of algebraic varieties which was introduced by two of the authors in \cite{HL}.
Thus, what we actually prove is Theorem \ref{themaintheorem} which is a version of the result above in this model-theoretic framework -- the final transfer to Berkovich spaces being straightforward. 

Let us give some explanations. Let $X$ be an algebraic variety over a valued field $F$.
We denote by $\widehat X$ the stable completion
of $X$. The standard skeleton $S_n$ of $(\gm^n)\an$
has a natural counterpart $\Sigma_n$
in $\widehat{\gm^n}$, 
and $\bigcup \phi_j^{-1}(\Sigma_n)$
makes sense as a subset of $\widehat X$;
moreover, the inequalities between log-regular functions
that cut out $S$ inside $\bigcup \phi_j^{-1}(S_n)$ also
make sense here, and cut out a subset $\Sigma$ of 
 $\bigcup \phi_j^{-1}(\Sigma_n)$.
 By  Theorem \ref{theo-union-skeleta}, this subset is 
 $F$-definably homeomorphic to
 an $F$-definable subset of $\Gamma^N$
 for some $N$. It follows moreover from its construction 
 that $\Sigma$ is contained in the subset $X^\#$ of $\widehat X$
 consisting of strongly stably dominated types (or, in other words,
 of Abhyankar valuations), and even in
 its subset $X^{\#}_{\mathrm{gen}}$ of Zariski-generic points.
We can now state  Theorem \ref{themaintheorem}. 
Let us just precise that
what we call a \textit{val-rational} function 
is
a $\Gamma$-valued function of the form $\val (f)$ with $f$ a non-zero rational function on $X$ (here $\val(f)$ is seen as
defined on the stable completion of the invertibility locus of $f$.)

\begin{mtheo}[Model-theoretic setting]
Let $F$ be an  algebraically 
closed field endowed with a 
valuation
$\val : F \to \Gamma \cup \{\infty\}$. Let $X$ be an irreducible algebraic variety 
over $F$. Let $\Upsilon$ be an iso-definable
subset of $X^{\#}_{\mathrm{gen}}$ which is $\Gamma$-internal, that is, 
$F$-definably isomorphic to an $F$-definable
subset
of $\Gamma^N$ for some  $N$. 

There exists finitely many non-zero rational functions
$f_1,\ldots, f_\ell$ on $X$ such that the following holds. 

\begin{enumerate}[1]
\item The functions $\val (f_1),\ldots, \val ({f_\ell})$ identify
topologically $\Upsilon$
with an $F$-definable subset of $\Gamma^\ell$. 

\item The group of restrictions of val-rational functions to $\Upsilon$
is stable under $\min$ and $\max$ and generated
under addition, substraction, $\min$ and $\max$ by the (restrictions of the) functions $\val (f_i)$ and the constants $\val  (a)$ for
$a\in F^\times$.
\end{enumerate}
\end{mtheo}

Let us start with a remark.
The $\Gamma$-internal subsets
we are really interested
in for application
to Berkovich theory
seem to be of a very specific form
(they are definable subsets of
$\bigcup \phi_j^{-1}(\Sigma_n)$
for some family $(\phi_j)$
of quasi-finite maps from
$X$ to $\gm^n$)
and our main theorem
deals at first sight
with far more general 
$\Gamma$-internal subsets.
But this is somehow
delusive; indeed, we show
(Theorem \ref{structure-pure-gammainternal})
that every $\Gamma$-internal
subset of $X^\#_{\mathrm{gen}}$
is contained in some finite union
$\bigcup \phi_j^{-1}(\Sigma_n)$ as above. 

We are now going to describe roughly the main steps of
the proof of our main theorem. 

\subsubsection*{Step 1} This first step has nothing to do with valued fields and concerns 
general divisible abelian ordered groups. Basically, one proves the
following. Let $D$ be an $M$-definable
closed subset of $\Gamma^n$ for some divisible
ordered group $M$ contained in a model $\Gamma$
of $\DOAG$, let $g_1,\ldots, g_m$ be $\mathbb{Q}$-affine
$M$-definable functions on $\Gamma^n$, and let $f$ be any continuous
\textit{and Lipschitz} $M$-definable map from $D$ to $\Gamma$, such that 
for every $x$ in $D$ there is some index $i$ with $f(x)=g_i(x)$.
Then under
these assumptions, $f$ lies in the set of functions
from $D$ to $\Gamma$ generated
under addition, substraction, min and max by
the $g_i$, the coordinate functions 
and $M$: this is Theorem \ref{thm:lip-t-gen}. Here
the Lipschitz condition refers to a
Lipschitz constant in $\mathbb{Z}_{\geq 0}$, 
so that it is a void  condition when  $M$
has no non-trivial convex subgroup
and $D$ is definably compact, but meaningful in general.

\subsubsection*{Step 2} We start with proving a finiteness result in the spirit of our theorem under a weaker notion of generation. More precisely, we show
(Theorem \ref{theo-tfinite-pure})
the existence
of $f_1,\ldots, f_\ell$ as in our statement
such that (1) holds and such that the following weak version
of (2) holds, with $H$ denoting
the group of
$\Gamma$-valued functions
on $\Upsilon$ generated by the $\val (f_i)$ and the
constants $\val (a)$ for $a\in F^\times$ : \textit{for every
non-zero rational function $g$
on $X$ there exist finitely many
elements $h_1,\ldots, h_r$
of $H$ such that $\Upsilon$ is covered by its definable subsets $\{\val (g)=\val (h_i)\}$
for $i=1,\ldots, r$.}

The key point for this step is the  purely valuation-theoretic
fact that an Abhyankar extension of a defectless valued field is still defectless.
It has been given several proofs in the literature,
some of which are purely algebraic, some of which are
more geometric. For the sake of completeness and for consistency with the general viewpoint of this paper, we give a new one in  Appendix \ref{app:abh}, 
(Theorem \ref{theo-abhyankar-defectless}) which is model-theoretic and based upon \cite{HL}. 
It follows already from Theorem \ref{theo-tfinite-pure}  that skeleta are endowed with a canonical piecewise
$\mathbb{Z}$-affine structure. In particular this implies the existence of canonical volumes for skeleta as we spell out in  Section \ref{sec:volumes}.

\subsubsection*{Step 3}
One strengthens the statement of
Step 2 by showing (Proposition \ref{speandlip})
that the $f_i$  can even be chosen so that all
functions $(\val (g))|_\Upsilon$ as above are Lipschitz, when seen as functions
on $\val (f) (\Upsilon)\subset\Gamma^m$. 
This
is done as follows. First, by possibly replacing the ground field with a smaller one over which everything 
is defined, we can assume that $\val(F^\times)$ has only finitely many convex subgroups. Under
this assumption we can achieve by enlarging $f$
that $\val (f)$ induces an embedding $\Upsilon(F')\hookrightarrow \Gamma^m(F')$ for every coarsening $F'$ of $F$
(by a coarsening, we mean that $F'$ has the same underlying field as  $F$ and a coarser valuation); then
for every valued
algebraically closed extension $L$ of $F$ and every coarsening $L'$ of $L$ the map $\Upsilon(L')\to \Gamma(L')^n$
induced by $\val (f)$ will be injective, which implies the sought after Lipschitz property by an easy compactness argument.

\subsubsection*{Step 4}
One proves that the set of functions on $\Upsilon$ of the form $\val (g)$ is stable under $\min$ and $\max$. 
This follows from orthogonality between the residue field and the value group sorts
in $\acvf$, see Lemma \ref{lem-stable-min}. 

\subsubsection*{Step 5}
 By the very choice of the $f_i$, every function $\val(g)|_{\Upsilon}$ 
gives rise \textit{via} the embedding $\val(f)|_{\Upsilon}$ to a definable function on $\val (f)(\Upsilon)$ that
belongs piecewise to the group
generated by
$\val (F^\times)$ and the coordinate functions $x_1,\ldots,x_\ell$ (Step 2) and is moreover Lipschitz (Step 3); it is thus (Step 1) equal to
$t(x_1,\ldots, x_\ell,a)$ where $t$ is a term in $\{+,-,\min,\max\}$ and $a$ a tuple of elements of $\val(F^\times)$. Then $\val (g)|_{\Upsilon}
=t(\val(f_1)|_{\Upsilon}, \ldots, \val(f_\ell)|_{\Upsilon},a)$ and we are done.

\subsection*{Acknowledgements}The first and third authors  were partially supported  by ANR-15-CE40-0008 (D\'efig\'eo).
Part of this work was done
during a stay of the first author at Weizmann Institute, funded by a Jakob Michael visiting professorship.
The third author  was partially supported  by the Institut Universitaire de France.
The fourth author was partially supported  by  the Fondation Sciences Math\'ematiques de Paris and GeoMod AAPG2019 (ANR-DFG), \textit{Geometric and
Combinatorial Configurations in Model Theory}.
We are very grateful to Michael Temkin for communicating us the example in Remark \ref{contrex}.
We are also very grateful to the referees
for their careful reading and for their many comments that helped us to improve considerably the readability of the paper.


\section{Preliminaries}\label{section2}
\subsection{Stably dominated types}The aim of this section is to review some of the material from \cite{HL} that we will use in this paper.
The reader is refered to \cite{HL}  or to the surveys \cite{bbk}  or \cite{simons} for more detailed information.
In this paper, we shall work  in the framework of \cite{HL}, namely the theory $\acvf$ of 
algebraically closed valued fields $K$ with nontrivial valuation
in
the geometric language ${\mathcal L}_{\mathcal{G}}$
of
\cite{haskell-h-m2006}. We recall that this language is an extension
of the classical three-sorted language
with sorts $\VF$, $\Gamma$ and $\RES$ for the valued field, value group and residue field sorts, 
and additional symbols $\val$ and $\res$ for the valuation and residue maps, obtained by adding new sorts
$S_m$ and $T_m$, $m \geq 1$, corresponding respectively to 
lattices in $K^m$ and to the elements of the reduction of such lattices modulo the maximal ideal of the valuation ring.
By the main result of \cite{haskell-h-m2006} $\acvf$  has elimination of imaginaries in ${\mathcal L}_{\mathcal{G}}$.

Recall that in a theory $T$ admitting elimination of imaginaries in a given language $\mathcal{L}$,
for $M \models T$ and $A \subseteq M$, a type
$p (\overline{x})$ in $S_{\overline{x}} (M)$ is said to be $A$-definable if for
every
$\mathcal{L}$-formula
$\varphi (\overline{x}, \overline{y})$
there exists
an 
$\mathcal{L}_A$-formula
$d_p \varphi (\overline{y})$ such that
for every $\overline{b}$  in $M$,
$\varphi (\overline{x}, \overline{b}) \in p$ if and only if
$M \models d_p \varphi (\overline{b})$.
If $p \in S_{\overline{x}} (M)$ is definable via $d_p \varphi$, then the same scheme gives rise to a unique
type $p_{\vert N}$ for any elementary extension $N$ of $M$.
There is a general notion of stable domination for 
$A$-definable types: stably dominated types are in some sense ``controlled by their stable part''.
In the case of $\acvf$, there is concrete characterisation of  $A$-definable  stably dominated types as those which are
orthogonal to $\Gamma$, meaning that for every elementary extension $N$ of $M$, if $\overline{a} \models p_{\vert N}$, one has $\Gamma (N) = \Gamma (N \overline{a})$.

Let $X$ be an $A$-definable set in $\acvf$, with $A$ an ${\mathcal L}_{\mathcal{G}}$-structure. 
A basic result in ~\cite{HL} states that 
there exists a strict $A$-pro-definable set $\widehat{X}$ such that for any $C \supseteq A$, $\widehat{X} (C) $ is equal to the set of $C$-definable stably dominated types on $X$ (\cite[Theorem 3.1]{HL}).
Here by pro-definable we mean a pro-object in the category of definable sets and strict refers to the fact  that the transition morphisms can be chosen to be surjective.
Morphisms in the category of pro-definable sets are  called definable morphisms.

In fact $\widehat{X}$ can be endowed with a topology that makes it a pro-definable space in the sense of \cite[Section 3.3]{HL}.
In this setting there is a model theoretic version of compactness, namely definable compactness: 
a pro-definable space $X$ is said to be definably compact if every definable type on $X$ has a limit in $X$.
In an o-minimal structure $M$, this notion is equivalent to the usual one, namely a definable subset $X \subseteq M^n$ is definably compact if and only if it is closed and bounded.

\subsection{$\Gamma$-internal sets}Let us fix a valued field $k$ and a quasi-projective variety $X$ over $k$. We denote by $\Gamma$ the value group of $k$. The structure induced is that of an ordered abelian group in the language of ordered groups, in particular it is o-minimal. We extend $\Gamma$ to   $\Gamma_\infty = \Gamma \cup \{\infty\}$ with $\infty$ larger than any element of $\Gamma$.
A pro-definable set  is called iso-definable if it is pro-definably isomorphic to a definable set.
A $\Gamma$-internal subset $Z$ of $\widehat{X}$, or more generally of $\widehat{X \times \Gamma_\infty^m}$, is an iso-definable subset such that there exists a surjective definable morphism
$D \to Z$ (which can be assumed to be bijective by elimination of imaginaries)
with $D$ a definable subset of some $\Gamma_\infty^r$.

By ~\cite[Theorem 6.2.8]{HL}, 
if $Z$ is a $k$-iso-definable and $\Gamma$-internal subset of $\widehat{X}$,
there exists some finite $k$-definable set $w$ and a continuous injective definable morphism
$f : Z \hookrightarrow \Gamma_\infty^w$. In particular if $Z$ is definably compact such an $f$ is a homeomorphism onto its image.

\subsection{The Zariski-generic case}\label{gamma-internal-generic}
Assume that $k$ is algebraically closed.  We can then assume
$w = \{1, \ldots, n\}$. Then the definable injection $Z\hookrightarrow \Gamma_\infty^n$
alluded to above can be obtained by using (locally) valuations of regular
functions. Thus if $X$ is irreducible and $Z$ only consists of Zariski-dense points, we can find a dense open subset $U$
of $X$ and invertible functions
$g_1,\ldots, g_n$ on $U$ such that the functions $\val (g_i)$ induce a definable bijection between $Z$ and a $k$-definable subset of $\Gamma^n$
(without $\infty$). Moreover, by shrinking $U$ 
and adding some extra invertible functions to the $g_i$, we can assume that $g$ induces a closed immersion $U\hookrightarrow \gm^n$; then 
the functions $\val (g_i)$ induce a (definably) \textit{proper} map $\widehat U\to \Gamma^n$ and thus a definable
\textit{homeomorphism}
between $Z$ and its image.

\subsection{Retractions to skeleta}Since multiplication does not belong to the structure on the value group sort $\Gamma$, we have to consider generalized intervals, which are 
obtained by concatenating a finite number of (oriented) closed intervals in $\Gamma_\infty$.
Such a generalized interval $I$ has an origin $o_I$ and an end point $e_I$.

We may now define strong deformation retractions. Fix a valued field $k$ and a quasi-projective variety $X$ over $k$.
A strong deformation retraction of $\widehat{X}$ onto $\Upsilon \subseteq \widehat{X}$ is a continuous $k$-definable morphism
\[
H: I \times \widehat{X} \longrightarrow \widehat{X}
\]
such that
\begin{enumerate}
\item[$\bullet$] The restriction of $H$ to $\{o_I\} \times \widehat{X}$ is the identity on $\widehat{X}$.
\item[$\bullet$]  The restriction of $H$ to $I \times \Upsilon$ is the identity on $I \times \Upsilon$.
\item[$\bullet$] The image of the restriction $H_{e_I}$ of $H$ to $\{e_I\} \times \widehat{X}$ is contained in $\Upsilon$.
\item[$\bullet$]  For every $(t, a) \in I \times \widehat{X}$, $H_{e_I} (H(t, a)) = H_{e_I} (a)$.
\end{enumerate}

A special case of the main result of \cite{HL} states the following:

\begin{theo}\label{ret-theo}Let $X$ be a quasi-projective variety over a valued field $k$.
Then there is a ($k$-definable) strong deformation retraction
\[
H: I \times \widehat{X} \longrightarrow \widehat{X}
\]
onto a $\Gamma$-internal subset $\Upsilon \subseteq  \widehat{X}$ and a $k$-definable injection $\Upsilon
\to \Gamma_{\infty}^w$ for some finite definable set $w$, which is a homeomorphism onto its image and
such that for each irreducible component $W$ of $X$, $\Upsilon \cap \widehat{W}$ is of o-minimal dimension
$\dim (W)$ at each point.
\end{theo}

We shall call such a $\Gamma$-internal set $\Upsilon$ a \emph{retraction skeleton} of  $\widehat{X}$.
Note that this is what is called a skeleton in \cite{HL}, but we have decided to change the terminology to avoid conflict with the literature.

\begin{rema}
When $X$ is smooth and irreducible, 
there exists a deformation retraction as above with $\Upsilon$
consisting only of Zariski-generic points: this follows from the proof of Theorem 11.1.1 in \cite{HL}, see also Chapter 12 of \cite{HL};
so if $k$ is a model of $\acvf$ then $\Upsilon$ can be topologically and $k$-definably identified
with a subset of some $\Gamma^m$ by using valuations of non-zero rational functions
(\ref{gamma-internal-generic}).

Note that the smoothness assumption cannot be dropped for the above: if $X$ is a cubic nodal curve, any retraction skeleton $\Upsilon$ of 
$\widehat X$ contains the nodal point (and
any definable topological embedding
from $\Upsilon$ into some $\Gamma_\infty^w$
will send the nodal point to  a $w$-uple with at least one infinite coordinate). 
\end{rema}

\subsection{Strongly stably dominated types}In fact all retraction skeleta  of $\widehat{X}$ are contained in  the subspace $X^{\#} \subseteq \widehat{X}$ of strongly stably dominated types on $X$. 
The study of the space $X^{\#}$ is the subject of Chapter 8 of \cite{HL}. Loosely speaking
the notion of strongly stably dominated corresponds to a strong form of the Abhyankar property for valuations namely that the transcendence degrees of the extension and of the residue field extension coincide.
An important property of 
$X^{\#}$ is that it has a natural structure of
(strict) ind-definable subset of $\widehat{X}$. Furthermore, by ~\cite[Theorem 8.4.2]{HL},
$X^{\#}$ is exactly the union of all the retraction skeleta of $\widehat{X}$.

It seems plausible that arbitrary $\Gamma$-internal subsets of $\widehat X$ can be rather pathological, but those contained
in $X^{\#}$ should be reasonable. We shall see below that this is indeed the case at least
for $\Gamma$-internal subsets of $X^{\#}$ 
that consist of Zariski-generic points (when $X$ is irreducible). 
When $X$ is irreducible, we will denote by 
$X^{\#}_{\mathrm{gen}}$ the subset of  $X^{\#}$ consisting of Zariski-generic points.

\subsection{Connection with Berkovich spaces}

Let $k$ be a valued field with $\val (k) \subseteq \mathbb{R}_\infty$, which we assume to be complete.
Let $X$ be a separated and reduced algebraic variety of finite type over $k$.
Denote by $X^{\mathrm{an}}$ its analytification in the sense of Berkovich.
Chapter 14 of \cite{HL} is devoted to a detailed study of how one can deduce statements about
$X^{\mathrm{an}}$ from similar statements about $\widehat{X}$.
This comes from the fact that, if one denotes by $k^{\mathrm{max}}$ a 
maximally complete algebraically closed extension of $k$ with value group $\mathbb{R}$ and residue field the algebraic closure of the residue field of $k$,
there is a canonical and functorial
map
$\pi : \widehat{X}(k^{\mathrm{max}}) \to X^{\mathrm{an}}$ which is continuous, surjective, and closed. When $k = k^{\mathrm{max}}$, $\pi$ is actually a homeomorphism.
Furthermore, any  $k$-definable 
morphism $g : \widehat{X} \to \Gamma_\infty$ induces a unique map
$\tilde g : X^{\mathrm{an}} \to \mathbb{R}_\infty$ which is continuous if $g$ is, and
any ($k$-definable) strong deformation retraction
$H: I \times \widehat{X} \rightarrow \widehat{X}$
induced canonically a 
strong deformation retraction
$\tilde H: I (\mathbb{R}_\infty) \times X^{\mathrm{an}} \rightarrow X^{\mathrm{an}}$
compatible with $\pi$ for any $t \in I (\mathbb{R}_\infty)$.
Thus, if one defines a retraction skeleton $\Sigma$ in 
$X^{\mathrm{an}}$ as  the image under $\pi $ of the $k^{\mathrm{max}}$-points of a retraction skeleton in 
$\widehat{X}$, we obtain that when $X$ is quasi-projective there exists a
strong deformation retraction of
$X^{\mathrm{an}}$ onto a retraction skeleton $\Sigma$. Furthermore, 
the fact that retraction skeleta in $\widehat{X}$ are contained in $X^{\#}$
implies that 
any  point of $\Sigma$, as a type over $(k, \mathbb{R})$, extends to a unique stably dominated type; this type is strongly stably dominated and, restricted to $(k, \mathbb{R})$, it determines an Abhyankar extension of the valued field $k$, cf. Theorem 14.2.1 in \cite{HL}.

\section{Finite generation and Lipschitz functions in $\DOAG$}\label{section4}
In this section, we work in the theory of divisible ordered abelian groups which is denoted by $\DOAG$, and by definable we mean definable with parameters. We shall usually denote by
$\Gamma$ a model of $\DOAG$. We start with the definition of $w$-combination and $w$-generation.
\begin{defi}\label{def:t-gen}Let $X$ and $Y$ be definable topological spaces and $g$, $f_1,\ldots,f_n$ be definable continuous functions from $X$ to $Y$. We say $g$ is a \emph{$w$-combination} of $f_1,\ldots,f_n$ if for every $x\in X$, there is some $i\in \{1, \ldots, n\}$ such that $f_i(x)=g(x)$. Notationally, we use $[g=f_i]$ to denote the set $\{x\in X:g(x)=f_i(x)\}$. Hence, $g$ is a $w$-combination by $f_1,\ldots,f_n$ iff $X=\bigcup^n_{i=1}[g=f_i]$.
\end{defi}

In contrast, there is a stronger notion of combination that is very specific to $\DOAG$. 
\begin{defi}\label{def:l-gen}Let $X$ be a definable topological space and let $g$ and $f_i$, $i\in I$, be definable continuous functions $X\to \Gamma$. We say that $g$ is an \emph{$\ell$-combination} of the $f_i$ if $g$ lies in the $(\min,\max)$-lattice generated by $(f_i)_{i\in I}$. More explicitly, there are $f_1,\ldots,f_n$ in $(f_i)_{i\in I}$ such that $g$ is a function obtained by $f_1,\ldots,f_n$ and finitely many operations of $\min,\max$.
\end{defi}

We shall also use the following variants of  $w$ and $\ell$-combination.
\begin{defi}\label{def:l-genvar}Let $X$ be a definable topological space and let $g$ and $f_i$
be definable continuous functions $X\to \Gamma$ for $i\in I$. 
We say that $g$ is a \emph{$(w, +)$-combination} of  the $f_i$ if there exist $h_1, \ldots, h_n$ in the abelian group generated by the functions $f_i$, $i \in I$ such that $g$ is a $w$-combination of the $h_i$. 
We say that $g$ is an \emph{$(\ell, +)$-combination} of the $f_i$ if $g$ can be described by a formula involving only $+,-,\min$ and $\max$ and finitely many $f_i$.
\end{defi}

We say that a given set of functions containing the $f_i$ and stable under $w$-combination is $w$-generated by the $f_i$ if it consists precisely of the set of all $w$-combinations of the $f_i$. We define $(w,+)$, $\ell$ and $(\ell,+)$-generation in an analogous way. 

\begin{exem}\label{eg:min-k}
Let $X=\Gamma^n$ and $m_k:X\to \Gamma$ be the definable function which to $(x_1,\ldots,x_n)$ assigns the $k$-th smallest $x_i$. Clearly, $m_k$ is a $w$-combination of the coordinate functions $x_1,\ldots,x_n$. On the other hand, it is not hard to see that 
\[
m_k(x)=\min_{U\subseteq \{1,\ldots,n\},|U|=k} \max_{i\in U} x_i
\]
Hence the $m_k(x)$ are even $\ell$-combinations of $x_1,\ldots,x_n$.
\end{exem}
However, the two notions of combinations do not agree in general.
\begin{exem}\label{eg:dl}
Let $I$ be the interval $[0,\infty)\subseteq \mathbb Q$.
Let $D= I\times \{1,2\}\subseteq\mathbb Q^2$ and $f_1=0$,
$f_2=x_1$. Consider $g$ that is equal to
$f_i$ on $I\times \{i\}$ for $i=1,2$. Clearly
$g$ is a $w$-combination of $f_1$ and $f_2$.
However, we claim that $g$ is not an $(\ell, +)$-combination
of coordinate functions.
Indeed, if it were, then it would extend to a continuous
$\mathbb Q$-definable function $g'$ on $\mathbb Q^2$. Let $\Gamma$ be a model of $\DOAG$
containing $\mathbb Q$ and in which there is some $c>n$ for all $n\in\mathbb{N}$. Since $\tp(1,c)=\tp(\alpha,c)$ for any $1>\alpha>0$ and $g(1,c)=c$, so $g'(\alpha,c)=c$. However $g'(0,c)=g(0,c)=0$, in contradiction with the continuity of $g'$. For a connected version of this example, replace $D$ by $D'=D\bigcup \{0\}\times [1,2]$ and set $g=0$ on $\{0\}\times [1,2]$.
\end{exem}
This example suggests that interaction of the ambient space and the topology of $D$ plays a role in distinguishing the two notions of combinations. To proceed towards a topological characterisation for such properties, we need the following.
\begin{defi}\label{def:convex} Let $T$ be an o-minimal expansion of $\DOAG$ and $\Gamma\models T$ with $D\subseteq \Gamma^n$ definable. We say that $D$ is convex if for any $u$
and $v$ in $D$, $\frac{u+v}{2}\in D$.
\end{defi}
\begin{rema}\label{rmk:rcf}
When $T$ is an o-minimal expansion of the theory of real closed fields $\mathrm{RCF}$, this is equivalent to the usual definition of convexity for definable sets. For $u,v\in D$, let $L\subseteq [0,1]$ be $\{\alpha: \alpha u+(1-\alpha)v \in D\}$. By our notion of convexity, $L$ contains $\mathbb{Z}[1/2]\cap [0,1]$. By o-minimality, $L$ must be $[0,1]$ with at most finitely many points in $(0,1)$ removed. But  removing any point from $(0,1)$ would lead to  a violation of convexity.

Note further that for $D$ convex, working inside the smallest affine subspace containing $D$, we may assume that $\cl(\mathrm{int}(D))=\cl(D)$.
\end{rema}

Lastly, recall that for any definable subset $D$ of some $\Gamma^n$, a function $f:D\to \Gamma$ is called \emph{$\mathbb{Q}$-affine} if $f=\sum^n_{i=1} m_ix_i+c$ where $m_i\in \mathbb{Q}$ and $c\in \Gamma$. Such functions are the most basic definable continuous functions on $D$. We say $f$ is \emph{$\mathbb{Z}$-affine} if the $m_i$ are all in $\mathbb{Z}$.
\begin{prop}\label{prop:t-l-equiv}
Let $\Gamma$ be a divisible ordered abelian group and let $f_1,\ldots,f_m$ be $\mathbb{Q}$-affine functions on $\Gamma^n$. Let $D\subseteq \Gamma^n$ be definable and $g:D\to \Gamma$ be a continuous definable function. Assume that $g$ is a $w$-combination of $f_1,\ldots,f_m$. Then the following are equivalent:
\begin{enumerate}
    \item $g$ is an $\ell$-combination of $f_1,\ldots,f_m$.
    \item $g$ extends to a continuous definable function $g':\Gamma^n\to \Gamma$ that is a $w$-combination of $f_1,\ldots,f_m$.
    \item $g$ extends to a continuous definable function $g':D'\to \Gamma$  on some convex definable set $D'$ containing $D$ that is a $w$-combination by $f_1,\ldots,f_m$.
    \item For any $x,y\in D$, there is $i \in \{1,\ldots,m\}$ such that $f_i(x)\leq g(x)$ and $g(y)\leq f_i(y)$.
    \item For some collection $S$ of subsets of $\{1,\ldots,m\}$, $g=\min_{X\in S} \max_{i \in X} f_i$.
\end{enumerate}
\end{prop}

\begin{proof}
The implications $(5)\implies (1)\implies (2)\implies (3)$ are clear.

For $(3)\implies (4)$, by working in an elementary extension, we may assume that $\Gamma$ is a model of the theory of real closed fields $\mathrm{RCF}$. By Remark~\ref{rmk:rcf} and after replacing $D$ by the convex set $D'$ in (3), we may assume the line segment $[x,y]$ connecting $x,y$ is in $D$. Replace $g$ by $g'$ given by (3) as well.  Let $I_j\subseteq [x,y]$ be $\{z: g(z)=f_j(z)\}$. By continuity of $g$ and o-minimality, we know that the sets $I_j$ are finite unions of closed intervals and $\bigcup^m_{j=1} I_j=[x,y]$. Consider the canonical parameterization $h:[0,1]\to [x,y], \alpha\mapsto \alpha y+(1-\alpha)x$, and let $f_i'=f_i\circ h$, $g'=g\circ h$ and $I'_j=h^{-1}(I_j)$. Since the functions $f_i$ are $\mathbb{Q}$-affine, the functions $f'_i$ are of the form $a_ix+b_i$ for some $a_i,b_i\in \Gamma$. Let $k$ be the $j$ such that $a_j$ is the greatest amongst all the $j$ such that $I'_j\neq \emptyset$. If there are multiple such $j$, pick any. By induction, for $a$ to the right of $I'_k$, we have $g'(a)\leq f'_k(a)$. Similarly, for $a$ to the left of $I'_k$, we have $f'_k(a)\leq g'(a)$. In particular we have $f'_k(0)=f_k(x)\leq g'(0)=g(x)$ and $g'(1)=g(y)\leq f_k(y)=f'_k(1)$.

For $(4)\implies (5)$, consider $S$ to be the collection of subsets $X\subseteq \{1,\ldots,m\}$ such that $g\leq \max_{i\in X} f_i$ on the entire $D$. 
Set $f:=\min_{X\in S}\max_{i\in X} f_i$.
We claim that $g= f$. Clearly $g\leq f$, so it suffices to show that $g\geq f$. For each $W\notin S$, there is some $y_W$ such that $g(y_W)>f_i(y_W)$ for
every $i \in W$. By $(4)$, for each $x\in D$, there is $i^x_{W}$ such that $f_{i^x_{W}}(x)\leq g(x)$ and $f_{i^x_W}(y_W)\geq g(y_W)$. Note that $i^x_W\notin W$. Let $X=\{i^x_W:W\notin S\}$.  We have that $X\in S$ because otherwise, $i^x_X\in X$. For this $x$, we have that $\max_{i\in X} f_i(x)\geq g(x)$ and $f_i(x)\leq g(x)$ for any $i\in X$, hence $f(x)\leq \max_{i\in X} f_i(x)=g(x)$.
\end{proof}

\begin{coro}\label{cor:cordinate}
Let $D\subseteq \Gamma^n$ be a definable convex set. The set of definable continuous functions from $D$
to $\Gamma$  is $(\ell,+)$-generated by the constants and all rational multiples of coordinate functions. 
\end{coro}
\begin{proof}
By quantifier elimination, we can find $\mathbb{Q}$-affine functions $f_1,\ldots,f_n$ such that $g$ is a $w$-combination of
 $f_1,\ldots,f_n$. By Proposition~\ref{prop:t-l-equiv}, we have that $g$ is in fact an $\ell$-combination of $f_1,\ldots,f_n$.
\end{proof}
Proposition~\ref{prop:t-l-equiv} suggests that the agreement of $w$-combination and $\ell$-combination is related to the existence of continuous extensions to an ambient convex space. This motivates the following definition.
\begin{defi}\label{def:Lipschitz}
For a tuple $x\in \Gamma^n$, define $|x|=\max^n_{i=1}|x_i|$. Let $D\subseteq\Gamma^n$ and $f:D\to \Gamma$ a definable function. We say $f$ is \emph{Lipschitz} if there is some $M\in \mathbb{N}$ such that $|f(x)-f(y)|\leq M|x-y|$. \end{defi}

Note that Lipschitz functions are automatically continuous and clearly the class of Lipschitz functions depends on the embedding of $D$ in $\Gamma^n$.
Our purpose is now to investigate
Lipschitz definable functions on closed definable sets; a first step
will consist in reducing to the definably compact case, 
by using the two following lemmas. 

\begin{lemm}\label{lem:lipschitz-closure}
Let $\Gamma$ be a model of $\DOAG$, let $D$
be a subset of $\Gamma^n$ definable over some set
$A$ of parameters, and let
$f\colon D\to \Gamma$ be a Lipschitz 
$A$-definable map. Let $(f_i)$ be a finite family
of $\mathbb{Q}$-affine $A$-definable functions 
such that $f$ is a $w$-combination
of the $f_i|_D$. Then $f$ admits a unique
continuous extension $\overline f$ to $\cl (D)$, 
the set $\cl (D)$ and the function $\overline f$
are $A$-definable, and $\overline f$ is Lipschitz
and if a $w$-combination of
the $f_i|_{\cl (D)}$. 
\end{lemm}

\begin{proof}
The uniqueness of $\overline f$ is clear, as well
as the $A$-definability of $\cl (D)$ and $\overline f$
if the latter exists, as one sees by
using the definition of the closure and of the limit
(with $\epsilon$ and $\delta$\ldots). The same reasoning
also shows that the set of points of $\cl (D) \setminus D$ at which $f$
admits a limit is $A$-definable. 
Moreover if $\overline f$ exists it inherits obviously the Lipschitz
property of $f$, and it is also $w$-generated by the (restrictions of)
the $f_i$: indeed, the subset of $\cl (D)$ consisting of points $x$
such that there is some $i$ with $f(x)=f_i(x)$ is closed and contains $D$, thus is the whole of $\cl (D)$. 

It thus remains to show the existence of $\overline f$,
and this can be done after enlarging the model $\Gamma$.
We can thus suppose that it is equal to the additive group of some
real closed field. 
Let $x$ be a point of $\cl (D) \setminus D$.
There exists a half-line $L$ emanating from
$x$ such that $(x,y)\subset D$ for some $y$;
taking $y$ close enough to $x$ we can assume that $f=f_j$ on
$(x,y)$ for some $j$. Then the limit of $f$
at $x$ along the direction of $L$ exists and is equal to $f_j(x)$. 
The Lipschitz property then ensures
that this limit does not depend on $L$,
let us
denote it by $\overline f(x)$. Since $D$
is defined by affine inequalities, there is a positive
$\gamma\in \Gamma$ such that for every $y$
in $\Gamma^n$ with $\|x-y\|<\gamma$
(say for the Euclidean norm) then either $(x,y)\subset D$
or $(x,y)\cap D=\emptyset$. Thus if $y$ is a point of $D$
with $\|x-y\|<\gamma$ then
$\abs{f(y)-\overline f(x)}\leq N\|x-y\|$
where $N$ is an upper bound for the slopes of the $f_i$. 
So $f(y)$ tends to $\overline f(x)$ when the point $y$
of $D$ tends to $x$.
\end{proof}

\begin{lemm} \label{lem:comp-l-gen}
Let $M$ be either $\{0\}$ or a model of $\DOAG$, let
$\Gamma$ be a model of $\DOAG$
containing $M$, and let $\rho$
be an element of $\Gamma$ 
with $\rho >M$. 
Let $Z\subseteq \Gamma^n$ be
an $M$-definable subset. Let $x_1,\ldots,x_n:Z \to \Gamma$ denote the coordinate functions of $Z$ and let $h:Z\to \Gamma$ be an
$M$-definable function. 

Assume that there 
exists a term $t$
in
$\{+,-,\max,\min\}$ and
$\gamma=(\gamma_1,\ldots,\gamma_l)$
in $\Gamma^\ell$
such that $h|_{Z_\rho}=t(x_1,\ldots,x_n,\gamma)|_{Z_\rho}$,
where $Z_\rho=Z\cap [-\rho,\rho]^n$. 

Then there is a term $t'$ in  $\{+,-,\max,\min\}$ and
a finite tuple $\beta$ of elements of $M$ such that $h=t'(x_1,\ldots,x_n,\beta)$.
\end{lemm}

\begin{proof}
Assume first that $M$ is 
a model of $\DOAG$. By our assumption,
there exists a term $t$ in $\{+,-,\max,\min\}$ and
a tuple
$\gamma=(\gamma_1,\ldots,\gamma_l)\in \Gamma^\ell$
such that $h|_{Z_\rho}=t(x_1,\ldots,x_n,\gamma)|_{Z_\rho}$. 
By model-completeness of $\DOAG$, the $\gamma_i$ can be chosen in
$M\oplus \mathbb Q\cdot \rho$.  
Thus there is $m>0$ such that for each $i$, there exist integers $k_i$ and $\beta_i\in M$
with $\gamma_i=\frac{k_i}{m}\rho+\beta_i$. Let $\nu$ denote $\rho/m$.
We have 
\[
h|_{Z_{m\nu}}=t(x_1,\ldots,x_n,(k_i\nu+\beta_i))|_{Z_{m\nu}}.
\]
Viewing the above as a first-order formula with constants in
the model $M$ and a variable for $\nu$, using o-minimality
and model-completeness of $M$, we have some $\nu_0\in M_{>0}$ such that for any $\nu'>\nu_0$, the following holds in $M$:
\[
h|_{Z_{m\nu'}}=t(x_1,\ldots,x_n,(k_i\nu'+\beta_i))|_{Z_{m\nu'}}.
\]
Take $\nu(x)=\max \{|x_1|,\ldots,|x_n|,2\nu_0\}$ and
\[
t'(x_1,\ldots,x_n,\beta)=t(x_1,\ldots,x_n,(k_i\nu(x)+\beta_i)).
\]
We then have 
\[
h=t'(x_1,\ldots,x_n,\beta)
\]
by construction, which ends the proof when $M\neq \{0\}$.

If $M=\{0\}$, 
set $\Gamma'=\Gamma\oplus \mathbb{Q}\cdot \delta$ where $\delta$ is positive
and infinitesimal with respect to $\Gamma$, set $M'=\mathbb{Q}\cdot \delta$ and
let us denote
by $Z'$ and $h'$ the objects
deduced from $Z$ and $h$ by base-change to $\Gamma'$. 
Applying the above yields
a
term 
$\theta$ in $\{+,-,\max,\min\}$ 
and a tuple $\beta$ 
of elements of $\mathbb{Q}\cdot \delta$ such that $h'=\theta(x_1,\ldots,x_n,\beta)$. 
By reducing modulo the convex subgroup $\mathbb{Q}\cdot \delta$ of $\Gamma'$
we see that $h=\theta(x_1,\ldots,x_n,0)$. 
\end{proof}

We can now state the main result of this section.

\begin{theo}\label{thm:lip-t-gen}
Let $M\models \DOAG$ or $M=\{0\}$ and
let $\Gamma$ be a model of
$\DOAG$ containing $M$.
Let $D\subseteq \Gamma^m$ be an $M$-definable set. Let $g:D\to \Gamma$ be a Lipschitz definable function over $M$. Let $f_1,\ldots,f_n$ be $\mathbb{Q}$-affine functions over $M$ such that $g$ is a $w$-combination of $f_1,\ldots,f_n$. Then $g$ is an $(\ell, +)$-combination of the $f_i$,
the constant $M$-valued functions and the coordinate functions. 
\end{theo}

Before proving this result we will need some preliminaries on cell decomposition in $\DOAG$. 

\subsection{Cell decomposition}
Fix a model $\Gamma$ of
$\DOAG$.
We shall use the notion
of special linear decompositions from ~\cite{elef}.  In ~\cite{elef}, Eleftheriou defines the notion of linear decomposition, which is a cell decomposition using only
graphs of $\mathbb{Q}$-affine functions instead of general piecewise $\mathbb{Q}$-affine functions. In fact we will need only to consider bounded linear cells in $\Gamma^n$. They are defined by induction on $n$.
In $\Gamma^0$ the origin is a bounded linear cell. If $C$ is a bounded linear cell in $\Gamma^{n-1}$, $f$ and $g$ are $\mathbb{Q}$-affine functions on $\Gamma^{n-1}$, with $f < g$ on $C$, the relative interval 
$(f < g)_C = \{(x', y) \in C \times \Gamma ; f (x') < y < g(x')\}$ and the graph $\Gamma (f)_C = \{(x', y) \in C \times \Gamma ; f (x') = y\}$ are bounded linear cells in $\Gamma^n$.
If $Y$ is a bounded definable set in $\Gamma^n$, a linear decomposition of $Y$ is a partition of $Y$ into (finitely many) bounded linear cells.

We denote by $\pi : \Gamma^n \to \Gamma^{n-1}$ the projection to the $n-1$ first coordinates.
A special linear decomposition of a bounded definable set $Y \subseteq \Gamma^n$ is defined recursively in ~\cite{elef} as follows. When $n = 1$ any cell decomposition of $Y$ is special.
If  $n > 1$, a linear decomposition $\mathcal{C}$ of $Y$ is special if the following conditions are satisfied:
 \begin{enumerate}[1]
\item $\pi (\mathcal{C})$ is a special linear decomposition of  $\pi (Y)$.
\item For every pair of  cells $\Gamma (f)_S$ and $\Gamma (g)_T$ in $C$ with $S$  in the closure of $T$,
$f|_S <g|_S$ or  $f|_S >g|_S$ or $f|_S =g|_S$.
\item\label{item3}  For every pair of cells $(f<g)_T$ and $X$ in $C$, where $X = \Gamma (h)_S$, $(h,k)_S$ or $(k,h)_S$,
there is no $c \in  \cl(S) \cap \cl(T)$ such that $f(c) < h(c) < g(c)$.
 \end{enumerate}
An important property of special linear decompositions is that if $D$ and $E$ are two cells in such a decomposition such that
$D \cap \cl (E)$ is non-empty then $D \subseteq \cl (E)$ (\cite{elef}, Fact 2.3).
By~\cite[Fact 2.2]{elef}  special linear decompositions  of $Y$ always exist.

Note that closures of cells have a simple description: 
the closure of $(f < g)_C$ is equal to $(f \leq  g)_{\cl (C)} =  \{(x', y) \in \cl (C) \times \Gamma ; f (x') \leq y \leq g(x')\}$ and the closure of 
$\Gamma (f)_C$, is  $\Gamma (f)_{\cl (C)}$.
In particular, if $C$ is a cell, $\pi (\cl (C)) = \cl (\pi (C))$.

\begin{lemm}\label{lemm:cell}Fix a special linear cell decomposition of a closed bounded definable subset of $\Gamma^n$ and let $C_1$ and $C_2$ be two cells. Set $D_1 = \cl (C_1)$ and
$D_2 = \cl (C_2)$. Assume that $D_1 \cap D_2$ is non-empty. Then there exists a cell $C$ such that
$D_1 \cap D_2 = \cl (C)$.
\end{lemm}

\begin{proof}We proceed by induction on $n$. 
The case $n=0$ is clear. If $n >0$, we have
that $\pi (D_1) \cap \pi (D_2) = \cl (C')$ for some cell $C'$ of the projection of the decomposition.
Since for $i = 1, 2$, $D_i\cap \pi^{-1} (C')$ is either of the form
$(f_i \leq  g_i)_{C'}$ or $\Gamma (f_i)_{C'}$, it follows from condition (3) of being a special linear decomposition  that either
$D_1 \cap D_2 \cap \pi^{-1} (C')
= (f_1\leq  g_1)_{C'}
= (f_ 2\leq  g_2)_{C'}$ or
$D_1 \cap D_2 \cap \pi^{-1} (C')
= \Gamma (f_1)_{C'} = \Gamma (f_2)_{C'}$, from which the statement follows.
\end{proof}

We shall also need the following statement.

\begin{lemm}\label{lemm:ineg}Let $D$ be a closed bounded definable subset of 
$\Gamma^n$. Assume $D$ is convex. Let $h$ be a $\mathbb{Q}$-affine function on $\Gamma^n$ such $h \geq 0$ on $D$.
Let $D_0$ be the zero locus of $h$ in $D$. We assume that $D_0$ is non-empty and $D_0 \neq D$.
Let $f$ be a $\mathbb{Q}$-affine function on $\Gamma^n$ which vanishes on $D_0$. Then there exists a positive integer
$M$ such that, for every $x \in D$, $\vert f (x) \vert \leq M h (x)$.
\end{lemm}

\begin{proof}Let $\mathcal{D}$ be a linear decomposition of $D$.
We consider the set $\mathcal{F}$ of all sets $F$ of the form $F = \cl (C)$, with $C$ a 1-dimensional cell in $\mathcal{D}$, 
that intersect the hyperplane $h =  0$ and are not contained in $h=0$.
For such an $F$ we denote by $p_F$ its intersection point with $h =  0$.
There exists a positive integer $M_F$ such that 
$\vert f (x) \vert \leq M_F h (x)$ on $F$. Indeed,
the restrictions of both $h$ and $g$ to the line segment $F$ are linear functions on $F$ vanishing at the endpoint $p_F$ of $F$ and the restriction of $h$ is not identically zero, which yields the existence of 
some $M_F$.
In fact the inequality $\vert f (x) \vert \leq M_F h (x)$ holds on the whole half-line $L_F$ containing $F$ with origin $p_F$.
Take $M = \max_{\mathcal{F}} (M_F)$. Now consider  $R$ a $\mathrm{RCF}$-expansion of $\Gamma$. Let $Y$ be the convex hull of the half-lines $L_F$ in that expansion. We have $\vert f (x) \vert \leq M h (x)$ on $Y$.
But $Y$ contains $D  (R)$,  since if $P$ is a convex definably compact polyhedron of $R^n$, and if $F$ is a face of $P$ of any dimension, then the convex hull of all half-lines directed by $1$-faces intersecting $F$ contains $P$,
hence the result, taking $P = D$ and $F = D_0$.
\end{proof}

The following statement about separation by hyperplanes will play a key role in our proof of Theorem \ref{thm:lip-t-gen}.

\begin{prop}\label{prop:cell}Fix a special linear cell decomposition of a closed bounded definable subset of $\Gamma^n$ and let $C_1$ and $C_2$ be two cells. 
Assume $C_1 \neq C_2$. Set $D_1 = \cl (C_1)$ and
$D_2 = \cl (C_2)$. Then there exists a $\mathbb{Z}$-affine function $h$ such that $h \geq 0$ on $D_1$, 
$h  \leq 0$ on $D_2$, and
the hyperplane $H =h^{-1}(0)$  satisfies $D_1\cap D_2 =D_1 \cap H=D_2\cap H$.
\end{prop}

\begin{proof}We shall proceed by induction on $n$, the case $n=1$ being clear.
If $\pi (C_1) = \pi (C_2)$, then the statement is clear. Indeed, for each $i = 1, 2$, we have
$C_i = (f_i < g_i)_S$  or $C_i = \Gamma(f_i)_S$. In the second case we set $g_i = f_i$.
We may assume that $C_1$ is above $C_2$. The graph of the average of $f_1$ and $g_2$ provides the required hyperplane.

Thus we will assume from now on that $\pi (C_1) \neq \pi (C_2)$. 
We set $C'_i = \pi (C_i)$ for $i = 1, 2$.
By Lemma \ref{lemm:cell}, if  $D_1 \cap D_2$ is non-empty, there exists a cell $C$ such that
$D_1 \cap D_2 = \cl (C)$.\\

\textbf{Case 1:} $D_1 \cap D_2$ is non-empty and $C$ is of the form $(f < g)_S$.\\

In this case, for $i = 1, 2$,  $C_i$ is necessarily  of the form
 $(f_i < g_i)_{C'_i}$ where  $f_i$ and $g_i$ are $\mathbb{Q}$-affine functions coinciding with $f$ and $g$ on $S$, since we are working with a special linear cell decomposition.
Furthermore  $D_i =  (f_i \leq g_i)_{\cl (C'_i)}$ and we have $f_1 = f_2$ and $g_1 = g_2$ on $\cl (C'_1) \cap \cl (C'_2)$.
It follows that $D_1 \cap D_2 = 
(f_i \leq g_i)_{\cl (C'_1) \cap \cl (C'_2)}$, for $i = 1,2$.
By the induction hypothesis, there exists an hyperplane $h'$ in $\Gamma^{n-1}$ given by a 
$\mathbb{Z}$-affine equation satisfying
the conditions of Proposition \ref {prop:cell} relatively to 
$\cl (C'_1)$ and $\cl (C'_2)$. Consider the vertical hyperplane $H$ above $h'$ (the hyperplane defined by the same equation in $\Gamma^n$).
It follows from our description of $D_1 \cap D_2$ that $H$ satisfies the required conditions.\\

\textbf{Case 2:} $D_1 \cap D_2$ is non-empty and  $C$ is of the form $\Gamma (f)_S$. \\

By the induction hypothesis, 
there exists an hyperplane $h'$  given by an equation $h' (x')=0$ in $\Gamma^{n-1}$, with $h'$ a $\mathbb{Z}$-affine function, $x'=(x_1,\ldots, x_{n-1})$
fulfilling
the conditions of Proposition \ref {prop:cell} relatively to 
$\cl (C'_1)$ and $\cl (C'_2)$. 
In particular $h' \geq 0 $ on 
$\pi (D_1)$ and $h' \leq 0$ on  $\pi (D_2)$.
We denote by $H$ the hyperplane with equation $h'(x')=0$ in $\Gamma^{n}$.

The set $C_1$ is of the form $(f_1 < g_1)_{C'_1}$ or $\Gamma (f_1)_{C'_1}$. In the second case we set $g_1 = f_1$.
Similarly $C_2$ is of the form $(f_2 <g_2)_{C'_2}$ or $\Gamma (f_2)_{C'_2}$ and in the second case we set $g_2 = f_2$.
Set $C' = \pi (C)$.
Since our linear decomposition  is special, 
we have that $f|_{C'}$ is equal to $f_1|_{C'}$ or $g_1|_{C'}$. Without loss of generality we may assume that $f|_{C'} = g_1|_{C'}$. It follows that  $f|_{C'}=f_2|_{C'}$ by the case assumption and the fact our decomposition is special. 
Let $X$ be the graph of $g_1$ over $\cl (C'_1)$. The function $x_n - f (x')$ is identically zero on
$H \cap X$, hence by Lemma \ref{lemm:ineg}, there exists a positive integer $M$ such that
$\vert x_n - f (x') \vert \leq M h' (x')$ on $X$. After increasing $M$ we may assume the inequality is strict when $h' (x') \neq 0$.
It follows that the 
the hyperplane $H_M$ with equation $Mh'(x’) - (x_n - f (x’)) = 0$ lies  above the set $D_1$
and strictly above $D_1 \setminus H$. Using the same argument for $D_2$, we get that after possibly increasing $M$
the hyperplane $H_M$  lies  under the set $D_2$ and strictly under $D_2 \setminus H$.
Let us check that $H_M$ satisfies the required conditions.
Indeed, a point $x = (x', x_n)$ lies in $D_1 \cap H_M$ if and only if
$x' \in \pi (D_1)$, $x \in H_M$, and  $f_1 (x') \leq x_n \leq g_1 (x')$.
But if $x \in D_1 \cap H_M$ we must have $h'(x') = 0$.
Thus 
$x$ lies in $D_1 \cap H_N$ if and only if
$x' \in \pi (D_1)$, $x \in H_M$, $h'(x')= 0$ and  $x_n = f(x')$, from which the equality
$D_1 \cap H_N = D_1 \cap D_2$ follows, and one gets similarly that 
$D_2 \cap H_N = D_1 \cap D_2$.\\

\textbf{Case 3:} $D_1 \cap D_2$ is empty. \\

If $\pi (D_1) \cap \pi (D_2) = \emptyset$ then by the induction hypothesis there exists an hyperplane 
$h'$ in $\Gamma^{n -1}$ satisfying the required conditions for $\pi (D_1)$ and $\pi (D_2)$ and $\pi^{-1} (h')$ will do the job.
Thus we may assume that $\pi (D_1) \neq \pi (D_2)$ and $\pi (D_1) \cap \pi (D_2) \neq \emptyset$.
 We choose an 
hyperplane 
$h'$  in $\Gamma^{n -1}$ with equation $h'(x')=0$ satisfying the required conditions  for $\pi (D_1)$ and $\pi (D_2)$.
We may assume $h' \geq 0$ on $D_1$ and $h' \leq 0$ on $D_2$.
As in Case 2,  $C_1$ is of the form $(f_1 <g_1)_{C'_1}$ or $\Gamma (f_1)_{C'_1}$. In the second case we set $g_1 = f_1$.
Similarly for $C_2$.

Set  $D'_1= D_1\cap H$ and $D'_2 = D_2 \cap H$. 
We have $D'_1 \cap D'_2 = \emptyset$. 
We may assume that $f_2 > g_1$ over $\pi (D_1) \cap \pi (D_2)$.
Note  that if we intersect the cells  of a special linear cell decomposition of some bounded set $W$ with $H$ we get a special linear cell decomposition of $W \cap H$.
Thus we can apply the induction hypothesis to $D'_1$ and $D'_2$, and there exists a  $\mathbb{Z}$-affine function $f$ on $\Gamma^n$
such that $f >0$ on $D'_1$ and $f <0$ on $D'_2$. 
We claim that for $M$ a large enough integer the hyperplane
$M h' +f  = 0$ will separate $D_1$ and $D_2$.

To prove this we proceed similarly as in the proof of Lemma \ref{lemm:ineg}.
We consider the set $\mathcal{F}$ of all sets $F$ of the form $F = \cl (C)$, with $C$ a 1-dimensional cell contained in $D_1$, 
that intersect  $H$ and are not contained in $H$.
For such an $F$ denote by
$p_F$ the intersection point of $F$ with $H$.
The restriction of $f$ to
$F$ can be written as $f (p_F) + \ell_F$ with $\ell_F$ a $\mathbb{Q}$-linear function on $F$.
Since $h'$ is strictly positive on $F$ outside $p_F$,
there exists a positive integer $M_F$ such that $M_F h' + \ell_F \geq 0 $ on $F$.
This still holds on the whole half-line $L_F$ containing $F$ with origin $p_F$. 
Since $f (p_F) > 0$ by assumption, we get that $M_F h'  + f  > 0$ on $L_F$.
Take $M_1 = \max_{\mathcal{F}} (M_F)$. 
Proceeding as  in the proof of Lemma \ref{lemm:ineg} we deduce that
$M_1 h'  + f >0$ on $D_1$.
One proves similarly the existence of $M_2$ such that
$M_2 h' + f <0$ on $D_2$. Thus we can take $M = \max (M_1, M_2)$.
\end{proof}

\begin{proof}[Proof of Theorem \ref{thm:lip-t-gen}]
By Lemma \ref{lem:lipschitz-closure}
we can assume that $D$ is closed. 
We may then enlarge the model $\Gamma$
and assume that it contains some 
$\rho$ with $\rho>M$. Let
$D_\rho$ denote the
intersection of $D$ with $[-\rho,\rho]^m$. 
This is a definably compact subset of $\Gamma^m$
which is definable over $M_\rho:=M\oplus \mathbb Q\cdot \rho$. 
If we prove that $g|_{D_\rho}$
is an $(\ell,+)$-combination of the $f_i$, the coordinate
functions and some constant functions with values in $M_\rho$,
Lemma \ref{lem:comp-l-gen}
above will allow us to conclude that $g$
is an $(\ell,+)$-combination of the $f_i$, the coordinate
functions and some constant functions with values in $M$.
We thus may and do assume that $D$ is definably compact.
By considering a
submodel of $M$ over which everything is defined, we reduce to
the case 
where
$M$ has exactly $r$ non-trivial convex subgroups,
and we proceed by induction
on $r$. The case $r=0$
is obvious since the definably compact set $D$ is then either empty 
or equal to
$\{0\}$. Assume now
that $r>0$ and that the result holds true
for smaller values of $r$. 

Let $M_0$ be the smallest non-trivial convex subgroup of $M$, and $\overline{M}=M/M_0$ be the quotient.
We are first going to explain why we can assume that $g(D(M))
\subseteq M_0$; this is tautological if $M_0=M$, so we assume 
(just for this reduction step) that $M_0\neq M$. 
In this case $\overline M$
is a model of $\DOAG$ with  $r-1$ non-trivial convex subgroups, and the natural map carrying $M$ to $\overline{M}$ induces a map that carries $D(M)$ to a definably compact
definable subset $\overline{D}$ of $\overline{M}^m$ (see~\cite[Theorem 4.1.1]{BP} for example). Furthermore, since $g$ is Lipschitz, it
 descends to a definable function $\overline{g}:\overline{D}(\overline{M})\to \overline{M}$, which is Lipshitz as well
 and is a $(w,+)$-combinations of the $\overline{f_i}$. 
By the induction hypothesis, we then know that $\overline{g}$ is of the form $\tau(\overline{f_1},\ldots,\overline{f_n})$, where $\tau$ is a term involving constants, projections and $+,-,\min,\max$ only. Replacing $g$ by $g-\tau(f_1,\ldots,f_n)$, we may assume that  $g(D)\subseteq M_0$, as announced.

By~\cite[Fact 2.2]{elef} there exists  a special linear decomposition  $\mathcal{D}$  of $D$ such that $g$ is $\mathbb{Q}$-affine on each cell. 
Clearly $D$ is covered by the closed sets $D_i = \cl (C_i)$, for $C_i$ in $\mathcal{D}$. 
In fact if one considers the set
$\mathcal{D}'$ of all $C\in \mathcal{D}$ such that, for any $C'\neq C$, $C$ is not contained in the closure of $C'$, it follows from
\cite[Fact 2.3]{elef} that the closed sets   $D_i = \cl (C_i)$ for 
$C_i$ in $\mathcal{D}'$ already cover $D$, but we will not use this. Sets of the form $\cl (C)$ with $C \in \mathcal{D}$ will be refered to as closed cells.

We will now use the separating hyperplanes provided by Proposition \ref{prop:cell}
to build affine functions that will appear in the $(\ell,+)$-combination we are seeking for describing $g$. For this purpose, 
the inclusion $g(D(M))\subseteq M_0$ will be crucial.

\begin{claim}\label{claim:separation}
Let $C'$ and $C''$ be any two distinct cells in $\mathcal{D}$. Set $D' = \cl (C')$ and $D'' = \cl (C'')$.
There exists a function $f_{D',D''}$ in the group generated by $f_1,\ldots,f_n$, the constant functions and the coordinate functions such that
\begin{equation}\label{eqn:separation}\tag{$*$}
    g|_{D'}\leq  f_{D',D''}|_{D'} \text{ and } f_{D',D''}|_{D''}\leq g|_{D''}.
\end{equation}
\end{claim}

\begin{proof}[Proof of the Claim]
By Proposition~\ref{prop:cell} there exists a $\mathbb{Z}$-affine function $h$ such that the hyperplane $H =h^{-1}(0)$  satisfies $D' \cap D'' =D' \cap H=D''\cap H$, $h \geq 0$ on $D'$ and
$h  \leq 0$ on $D''$.

If $D'\cap D''=\emptyset$, using definable compactness of $D'$ and $D''$, we get that there exists  $a \in M_0$ such that $h|_{D''}<-a<0<a<h|_{D'}$.  
Moreover, by our assumption that $g(D(M)))\subseteq M_0$, there is $b\in M_0$ such that $g(D (M)) \subseteq (-b, b)$. 
For any positive integer $m$ we have $mh - g > ma - b$ on $D'$ and  $mh - g < -ma + b$ on $D''$.
Since $M_0$ is archimedean,  for $m$ large enough, we have $ma > b$, hence
condition~(\ref{eqn:separation})  is satisfied for $f_{D',D''} = mh$.

If $D'\cap D''\neq \emptyset$, take $c\in D'\cap D''$ and  let $G$ be the $\mathbb{Q}$-affine function such that $g=G$ on $D'$. 
Replacing $g$ by $g-G$, we may assume that $g=0$ on $D'$. Translating our entire set by $c$, we may assume that $c$ is the origin. Thus $g(0)=0$
and $g$ is actually the restriction of a $\mathbb{Q}$-linear function on $D''$. 
On $D'$, for any positive integer $m$, we have $mh\geq0=g$. 
For any $b\in D''$, if $h(b)=0$, then $b\in D''\cap H=D'\cap H$, hence $g(b)=0$. 
Thus, by Lemma \ref{lemm:ineg}, there exists a positive integer $m$ such that
$-g \leq -mh$ on $D''$. For such an integer $m$, we have $g \leq mh$ on $D'$ and $g \geq mh$ on $D''$.
\end{proof}

We can now  conclude the proof of Theorem \ref{thm:lip-t-gen}. Note that $g$ is a $w$-combination of the functions $f_i$;  it is thus a fortiori a $w$-combination of
the set of functions obtained by adding all the functions $f_{D',D''}$ from Claim \ref{claim:separation} to the functions $f_i$.
Take $x$ and $y$ in $D$. If they belong to the same closed cell $D' = \cl(C')$, then $g (x) = f_i (x)$
and $g (y) = f_i (y)$ for some $i$ and condition (4) in Proposition~\ref{prop:t-l-equiv} is satisfied.
If they belong to two distinct closed cells $D'$ and $D''$,
then $g (x)  \leq  f_{D',D''} (x)$ and $ f_{D',D''} (y) \leq g (y)$ by Claim \ref{claim:separation}.
Thus, by the implication (4) $\implies$ (1) in 
Proposition~\ref{prop:t-l-equiv}, we obtain that $g$ is an $\ell$-combination of
the functions $f_i$ and $f_{D',D''}$, which concludes the proof.
\end{proof}
The proof of Claim~\ref{claim:separation} actually yields the following convenient way to check if a given function is Lipschitz on a definably compact set.

\begin{coro}
Let $D'$ and $D''$ be two definably compact convex sets such that $D'\cap D''=D'\cap H_a=D''\cap H_a\neq \emptyset$ with $H_a$ an hyperplane defined by a $\mathbb{Z}$-affine function. Assume further that  $g$ a continuous function that is affine on $D'$ and $D''$ respectively, then $g$ is Lipschitz on $D'\cup D''$.
\end{coro}

\begin{rema}
Note that one can have definably compact versions of Example~\ref{eg:dl} by replacing $[0,\infty)$ with $[0,c]$ for some $c>n\cdot 1$ for all $n\in \mathbb{N}$. However, the function $g$ there is not Lipschitz because $|(0,c)-(1,c)|=1$ and $|g(0,c)-g(1,c)|=c$.
\end{rema}

\begin{rema}In the case of homogeneous linear equations, with no parameters, equivalence of  $\ell$-combination and $w$-combination goes back to work of Beynon \cite{beynon}, see also \S5.2 of \cite{glass} and \cite{Ovchinnikov} for related results. In 2011, as a student, Daniel Lowengrub rediscovered and partially generalized Beynon's results. He also gave Example \ref{eg:dl} showing that they do not hold over non-archimedean parameters.   Here we fully generalized them, after replacing continuity by a Lipschitz condition.
Our  proofs in this section make use of his ideas.
\end{rema}

\section{Complements about $\Gamma$-internal sets}

\subsection{Preimages of the standard $\Gamma$-internal subset of
$\widehat{\gm^n}$}\label{ss-preimsn}
Let $k$ be an algebraically
closed valued field
and let $X$ be an irreducible
$n$-dimensional $k$-scheme of finite type.

Let
$\Sigma_n$ be the image of the definable topological embedding from $\Gamma^n$ into $\widehat{\gm^n}$
that sends a $n$-tuple $\gamma$ to the generic point $r_\gamma$ of the closed $n$-ball with valuative radius $\gamma$ and centered at the origin.
This set
$\Sigma_n$ is the archetypal example of a $\Gamma$-internal subset, and it is contained in $(\gm^n)^{\#}_{\mathrm{gen}}$. 

Let $\phi$ be any morphism from $X$ to $\gm^n$. Set
$\Upsilon=\phi^{-1}(\Sigma_n)$. 
If $\dim\phi(X)<n$ then $\phi(\widehat X)$ does not meet $\Sigma_n$ (since the latter lies over
the generic point of $\gm^n$), so $\Upsilon=\emptyset$. 
Assume that $\dim \phi(X)=n$, which means that $\phi$ is generically finite. 
Then each point of $\Upsilon$ lies in $X^\#_{\mathrm{gen}}$ by Proposition 8.1.2 in \cite{HL}
and $\phi^{-1}(s)$ is finite for every $s\in \Sigma_n$.

The purpose of what follows is to show that $\Upsilon$ 
is $\Gamma$-internal and purely $n$-dimensional, and that this 
also holds more generally for a finite union of pre-images
for $\Sigma_n$ under various
maps from
$X\to \gm^n$. 
This is a model-theoretic version of a result that is known in the Berkovich setting, see \cite{confluentes}, Theorem 5.1.

\begin{theo}\label{theo-union-skeleta}
Let $X$ be an $n$-dimensional $k$-scheme of finite type and let $\phi_1,\ldots, \phi_m$ be morphisms
from $X$ to $\gm^n$. The finite union 
$\Upsilon:=\bigcup \phi_j^{-1}(\Sigma_n)$ is a purely $n$-dimensional 
$\Gamma$-internal subset of $\widehat X$ contained in $X^\#_{\mathrm{gen}}$. 
\end{theo}

\begin{proof}[Proof of Theorem \ref{theo-union-skeleta}]
It is sufficient to prove that $\phi_j^{-1}(\Sigma_n)$
is $\Gamma$-internal and purely $n$-dimensional for every $j$. 
Indeed, assume that this is the case. Then 
 if $j$ and $\ell$ are two indices the intersection 
$\phi_j^{-1}(\Sigma_n)\cap \phi_\ell^{-1}(\Sigma_n)$ is definable in both $\phi_j^{-1}(\Sigma_n)$ and $\phi_\ell^{-1}(\Sigma_n)$ 
by \cite[Lemma 8.2.9]{HL} so $\Upsilon$ is $\Gamma$-internal, and obviously purely $n$-dimensional
as a finite union of purely $n$-dimensional $\Gamma$-internal subsets. 

We can thus assume that $m=1$, and we write $\phi$ instead of $\phi_1$. 
By its very definition, $\Upsilon$ is pro-definable, and we have seen above that it
is contained in the strict ind-definable set $X^\#$. 
It lies therefore inside a definable subset of $X^\#$, and by using once again  \cite[Lemma 8.2.9]{HL}
we see that $\Upsilon$ is iso-definable. 
Moreover we also have seen above that $\Upsilon\to \Sigma_n$ has finite fibers, thus using
\cite[Corollary 2.8.4]{HL} or
the fact that for any tuple $a$ of elements of $\Gamma$ the algebraic and definable
closures of $a$ over $k$ coincide (\cite[Lemma 3.4.12]{haskell-h-m2006}),
one deduces that
the definable set $\Upsilon$ is $\Gamma$-internal since $\Sigma_n$ is.

It remains to show that it is purely $n$-dimensional.
Since $\Upsilon$ is contained in $X^\#$, and lies over the quasi-finite locus $V$ of $\phi$,
it is contained in $\widehat U$ for any Zariski-open subset $U$
of $V$ meeting all $n$-dimensional components of $V$; this holds in particular for $U$ the flat locus of
$\phi|_V$. The flatness of the map $U\to \gm^n$ implies that $\widehat U\to \widehat{\gm^n}$
is open by \cite[Corollary 9.7.2]{HL}, so the finite-to-one map $\Upsilon\to \Sigma_n$ is open. As a consequence
$\Upsilon$ is purely $n$-dimensional. 
\end{proof}

Our purpose is now to prove that conversely,
every
$\Gamma$-internal subset of $X^\#_{\mathrm{gen}}$
is contained in some finite union
$\bigcup_j \phi_j^{-1}(\Sigma_n)$ as above
(Theorem \ref{structure-pure-gammainternal}); 
this is an instance of the general principle according to
which $\Gamma$-internal subsets
of $X^\#$ are expected to be reasonable
(while general $\Gamma$-internal subsets
of $\widehat X$ can likely
be rather pathological). 
Originally we used this result 
through Corollary 
\ref{closure-gamma-internal}
for proving Theorem
\ref{themaintheorem}, 
but we finally do not need it anymore. 
Nonetheless, we have chosen to keep it in this paper, 
because it seems to us
of independent interest, and shows that the main
objects considered
in this work are more tractable
than one could
think at first sight.

We start with a result which will be used for proving our theorem but 
is of independent interest; this is the analogue of 
\cite{confluentes}, Theorem 3.4 (1). If $x$ is a point of $\widehat X$ and if $f=(f_1,\ldots, f_n)\colon X\to \gm^n$
is a morphism, the \emph{tropical} dimension of $f$ at $x$ is the infimum of 
$\dim \val (f) (\widehat V)=\dim \val (f)(V)$ for $V$ an arbitrary  definable subset of $X$ 
such that $\widehat V$ is a neighborhood of $x$ in $\widehat X$. 

\begin{prop}\label{prop-dimtrop}
Let $f=(f_1,\ldots, f_n) \colon X\to \gm^n$ be a morphism,
and set $\Upsilon=f^{-1}(\Sigma_n)$. Then
$\Upsilon$ is exactly the set of points of $\widehat X$ at which the tropical dimension of $f$
is equal to $n$.
\end{prop}

\begin{proof}
Let $x\in \widehat X$. A point $x$ of $\widehat X$ belongs to $\Upsilon$ if and only if 
$f_1,\ldots, f_n$ is an Abhyankar basis at $x$, \ie 
\[\val\left(\sum a_I f^I(x)\right) =\min_I \val (a_I)+\val (f^I(x))\] 
for any non-zero polynomial $\sum a_I T^I$ with coefficients in $K$.

Now let $x\in \widehat X\setminus \Upsilon$. Then $f_1,\ldots, f_n$ is \emph{not}
an Abhyankar basis at $x$. Therefore there exists a polynomial $\sum a_I T^I$ with coefficients in $K$ such that 
\[\val\left(\sum a_I f^I(x)\right) >\min_I \val (a_I)+\val (f^I(x)).\] 
Let $V$ be the subset of $X$ defined by the inequality 
\[\val\left(\sum a_I f^I\right)>\min_I \val (a_I)+\val (f^I).\]
It is a definable subset of $X$, and its stable completion is an open
neighborhood of $x$ in $\widehat X$. 
Moreover by the very definition of $V$, for every $y\in V$ there exists two distinct
multi-indices $I$ and $J$ with $\val (a_I)+\val (f^I(y))=\val (a_J)+\val (f^J(y))$, which shows that
$\val (f) (V)$ is contained in a finite union of $(n-1)$-dimensional subspaces of $\Gamma^n$. 
As a consequence, the tropical dimension of $f$ at $x$ is at most $n-1$.

Conversely, let $x\in \Upsilon$ and let $V$ be a definable open subset of $X$ such that $\widehat V$
is a neighborhood of $x$ in $\widehat X$. 
Since $x\in \Upsilon$, it is contained in $X^\#_{\mathrm{gen}}$. 
There is a dense open subset $U$ of $X$ such that $f$ induces a finite flat map from $U$ to a dense
open subscheme of $\gm^n$; then the induced
map $\widehat U\to \widehat {\gm^n}$ is open
by \cite[Corollary 9.7.2]{HL}, and since $x\in X^\#_{\mathrm{gen}}$, it belongs
to $\widehat U$; as a consequence, $f$
is open around $x$. In particular, $f(\widehat V)$ contains a neighborhood $\Omega$ of $f(x)$. 
Since $x\in \Upsilon$, the image $f(x)$ is equal to $r_\gamma$ for some $\gamma \in \Gamma^n$. 
The intersection $\Omega \cap \Sigma_n$ then contains $\{r_\delta\}_{\delta \in B}$ for $B$ some product of $n$ open intervals
containing $\gamma$, so $\val (f)(\widehat V)$ contains $B$, and is in particular $n$-dimensional. 
The tropical dimension of $f$ at $x$ is thus equal to $n$. 
\end{proof}

We are now ready to establish the announced description
of a $\Gamma$-internal subset $\Upsilon$ of $X^{\#}_{\mathrm{gen}}$. 
The case where $\Upsilon$ is purely $n$-dimensional will rely of the description
of the maximal tropical dimension locus given by the above proposition. 
The general case will then be handled 
by embedding $\Upsilon$ into a purely $n$-dimensional $\Gamma$-internal subset of $X^{\#}_{\mathrm{gen}}$ --
the basic idea for doing this is to increase the dimension of $\Upsilon$ (until $n$ is achieved) by
``following" it along a deformation retraction as built in \cite{HL}.

\begin{theo}\label{structure-pure-gammainternal}
Let $X$ be
an $n$-dimensional integral scheme of finite type over $k$,
and let $\Upsilon\subseteq X^\#_{\mathrm{gen}}$ be
a $\Gamma$-internal subset defined over $k$. 
There exists a dense open subset $U$
of $X$ and finitely many morphisms  $\phi_1,\ldots, \phi_m$ 
from $U$ to $\gm^n$ such that $\Upsilon\subseteq\bigcup_j \phi_j^{-1}(\Sigma_n)$. 
\end{theo}

\begin{proof}
Let us first assume that $\Upsilon$ is purely $n$-dimensional. 
Since $k$ is algebraically closed, after shrinking $X$ we might assume that there exist
finitely many invertible functions $f_1,\ldots, f_r$ on $X$ such that $\val (f)$ induces a $k$-definable
homeomorphism between $\Upsilon$ and a definable subset of $\Gamma^r$
(\ref{gamma-internal-generic}). For every
subset $I$ of $\{1,\ldots, r\}$
of cardinality $n$, let $f_I$ be the map from $X$ to $\gm^I$ given by the $f_i$ with $i\in I$. Since $\Upsilon$
is of pure dimension $n$, for every $x\in \Upsilon$ there is at least one subset $I$ of $\{1,\ldots, r\}$
of cardinality $n$ such that the tropical dimension of $f_I$ at $x$ is $n$. By Proposition \ref{prop-dimtrop},
this means that \[\Upsilon \subseteq\bigcup_{I\subseteq\{1,\ldots, r\}, |I|=n}
f_I^{-1}(\Sigma_n),\]
which ends the proof in this particular case. As a by-product, we get in view of 
Theorem
\ref{theo-union-skeleta} that a finite union of purely $n$-dimensional $\Gamma$-internal subset
of $X^\#_{\mathrm{gen}}$ is still $\Gamma$-internal (and of course purely $n$-dimensional).

Let us now go back to an arbitrary $\Upsilon$. In order to prove the theorem, 
it suffices by the above to show that $\Upsilon$ is contained in some
purely $n$-dimensional $\Gamma$-internal subset
of $X^\#_{\mathrm{gen}}$. 
By shrinking $X$ we can assume
that it is quasi-projective. We have already noticed
that a 
finite union of purely $n$-dimensional $\Gamma$-internal subset
of $X^\#_{\mathrm{gen}}$ is still $\Gamma$-internal and purely $n$-dimensional, 
which allows ourselves to cut $\Upsilon$ into finitely many $k$-definable pieces and to argue piecewise. 
We thus can assume that $\Upsilon$ is purely $d$-dimensional for some $d$, and we argue by descending induction on $d$, so
we assume
that our statement holds if the $\Gamma$-internal subset
involved is  equidimensional of dimension
$>d$. 

Let $\alpha$ be
a $k$-definable embedding from $\Upsilon$ into some $\Gamma^m$ given by finitely many non-zero rational functions. 
By~\cite[Theorem 11.1.1]{HL}, there exists a pro-definable deformation retraction $h:I\times \widehat{X}\to\widehat{X}$ preserving $\alpha$ with a
$\Gamma$-internal purely $n$-dimensional image $\Upsilon_{\mathrm{targ}}$ contained in $X^\#.$
Let $\Upsilon_s =\{p\in \Upsilon: h(t,p)=p \text{ for any } t\}$. By its very definition, $\Upsilon_s$ is contained in the set
$\Upsilon'$ of Zariski-dense
points of $\Upsilon_{\mathrm{targ}}$, which is a purely $n$-dimensional $\Gamma$-internal subset of $X^{\#}_{\mathrm{gen}}$. It
therefore suffices to prove the proposition for the open complement of $\Upsilon_s$
in $\Upsilon$, which is still purely $d$-dimensional. 
In other words, we can assume that $\Upsilon_s=\emptyset$. 

Let $\Upsilon''=h(I,\Upsilon)$. We claim that it is iso-definable, and thus $\Gamma$-internal.
By~\cite[Lemma 2.2.8]{HL}, $\Upsilon''$ is strict pro-definable.
Since  $\Upsilon\subset
X^\#$, the set $\Upsilon''$
is contained in $X^\#$
as well  by~\cite[Theorem 11.1.1]{HL}
and the latter is strict ind-definable.
Hence by compactness, we see that $h(I,\Upsilon)$ is a strict pro-definable subset of a definable set, thus is iso-definable. 
Note also that the homotopy built in \cite{HL} is Zariski-generalizing, so $\Upsilon''\subseteq X^\#_{\mathrm{gen}}$.

Since $\Upsilon_s=\emptyset$, for every $p\in \Upsilon$ there are some $a_p, b_p$ in $I$ with $a_p<b_p$  such that
$h|_{[a_p,b_p]}: [a_p,b_p]\to \widehat{X}$ is injective.
Since $\Upsilon''$ is $\Gamma$-internal, the induced function
$h:I\times \Upsilon\to \Upsilon''$ is a definable function in the o-minimal sense.
Let $x=h(p,t)$ be a point of $\Upsilon''$, with $t$ and $p$ defined over $k$. 
We claim that $\dim_p \Upsilon''=d+1$. Indeed,
since $\dim_p \Upsilon=d$, there exists a point $q$ in $U$ that specialises to $p$ (when viewed as a type over $k$)
and such that $\alpha(q)$ is
$d$-dimensional over $k$ (\ie, its coordinates generate a group
of rational rank $d$ over $\Gamma(k)$). 
Now up to replacing
$t$ if necessary by an
endpoint of an interval containing $t$
on which $h(p,\cdot)$ is constant,
we may assume that there exists
a
non-singleton segment $J\subseteq I$ having $t$ as
one of its endpoints
such that $h(p,\cdot)|_J$ is injective. 
If $K$ is some subinterval
of $J$ containing $t$ defined over $k(q)$ and on which
$h(q,\cdot)$ is constant then since $h$ is continuous and
thus is compatible with specialisation, both endpoints
of $K$ have to specialise to $t$. Thus there exists a
non-singleton
segment $K$ contained
in $J$ and defined over $k(q)$,
having one endpoint  $\tau$
that specialises to $t$, on which $h(q,\cdot)$ is injective.
Now let us choose an element $\tau'$ of $K$ that specialises 
to $t$ and such that $k(\tau')$ is of dimension 1 over $k(q)$. 
By construction
$h(q,\tau')$ is a point
of $\Upsilon''$ that specialises
to $h(p,t)$ and that is $(d+1)$-dimensional over $k$, whence our claim. 

It follows that $\Upsilon''$ is of pure dimension $d+1$, and it contains $\Upsilon$. By induction $\Upsilon''$ is contained in some 
purely $n$-dimensional $\Gamma$-internal subset of $X^{\#}_{\mathrm{gen}}$, and we are done. 
\end{proof}

This theorem has an interesting consequence concerning the closure
$\overline{\Upsilon}$ of $\Upsilon$,
or at least its subset $\overline{\Upsilon}_{\mathrm{gen}}$ 
consisting of Zariski-generic points (let us mention that the general structure of the closure of an arbitrary $\Gamma$-internal subset is poorly understood). 

\begin{coro}\label{closure-gamma-internal}
Let $X$ be
an $n$-dimensional integral scheme of finite type over $k$,
and let $\Upsilon\subseteq X^\#_{\mathrm{gen}}$ be
a $\Gamma$-internal subset.
The set  $\overline{\Upsilon}_{\mathrm{gen}}$ 
is contained in $X^\#$ and is $\Gamma$-internal. 
\end{coro}

\section{A first finiteness result}\label{t-finite-pure}\label{section3}

The aim in this section is to prove 
a finiteness result, Theorem \ref{theo-tfinite-pure}, which is weaker
than our main theorem but will be needed in its proof. 

\subsection{Notation}
Throughout this section we fix a valued field $k$, an
$n$-dimensional integral $k$-scheme of finite type $X$, and a $\Gamma$-internal subset 
$\Upsilon$ of $X^{\#}_{\mathrm{gen}}$. Every non-zero
$k$-rational function $f\in k(X)$ gives rise to a $k$-definable map
$\val (f) \colon \Upsilon \to \Gamma$. The set of all such maps is denoted by $\SS_k(\Upsilon)$, or simply by $\SS(\Upsilon)$
if the ground field $k$ is clearly understood from the context. Elements
of $\SS(\Upsilon)$ will be called \textit{regular functions} from
$\Upsilon$ to $\Gamma$. 
By a \textit{constant} function on $\Upsilon$ we shall always mean
a \textit{$k$-definable} constant function; \ie, an element of $\val (k)\otimes \mathbb Q$.

Assume that $\val (k)$ is divisible, in which
case $\SS(\Upsilon)$ contains the constant functions.
We shall then say for short that $\SS(\Upsilon)$ is
finitely $(w,+)$-generated \textit{up to constant functions}
if there exist a finite subset $E$ of
$\SS(\Upsilon)$ such that $\SS(\Upsilon)$
is $(w,+)$ -generated by $E$
and the constant functions.

\begin{rema}
For a subset $E$ of
$\mathbb S(\Upsilon)$ to $w$-generate $\SS(\Upsilon)$, it suffices by compactness that for every $p\in \Gamma$
and every $f\in \SS(\Upsilon)$ there exists $g\in E$ such that $f(p)=g(p)$. 
\end{rema}

Our purpose is now to show that if $\val (k)$ is divisible
and $k$ is defectless, $\SS(\Upsilon)$ is finitely $(w,+)$-generated up to constant functions.
(Recall
that a valued field $F$
is called defectless or stable if
every finite extension of $F$ is defectless; to avoid any risk
of confusion with the model-theoretic
notion of stability  use the terminology defectless instead of stable.)
The core of the proof is the following proposition about valued field extensions.

\begin{prop}\label{prop-formula-valuegroup}
Let $F\hookrightarrow  K\hookrightarrow L$ be finitely generated extensions of
valued fields, with $K=F(a)$ and $L=K(b)$. 
We make the following assumptions: 
\begin{enumerate}[1]
\item $F$ is defectless; 
\item $K$ is Abhyankar over $F$; 
\item $\mathrm{res}(K)=\mathrm{res}(F)$; 
\item $\val (L)=\val(K)$;
\item $L$ is finite over $K$. 
\end{enumerate}

Then there exists a quantifier-free formula $\phi(x,y)$ in the language of valued fields
with parameters in $F$ such that $L\models \phi(a,b)$, and such that whenever $L'=F(a',b')$ is a valued
field extension with $L'\models \phi(a',b')$ and the residue field of
$K':=F(a')$ is a regular exension of $\mathrm{res}(F)$, then 
$\val(L')=\val(K')$.
\end{prop}

\begin{proof}
Since $F$ is defectless, $K$ is defectless as well (this was proved
by Kuhlmann in \cite{kuhl}, but for the reader's convenience we give a new proof of this fact in  Appendix \ref{app:abh}
with model-theoretic tools based upon \cite{HL}, see Theorem \ref{theo-abhyankar-defectless}). 
Therefore $L^{\mathrm h}$ is a defectless finite extension of $K^{\mathrm h}$; let 
$d$ denote its degree. By assumption one has $\val (L^{\mathrm h})=\val(K^{\mathrm h})$, so that $\mathrm{res}(L^{\mathrm h})$
is of degree $d$ over $\mathrm{res}(F^{\mathrm h})$. In other words, $\mathrm{res}(L)$ is of degree $d$ over $\mathrm{res}(K)$. 

Now let $c_1,\ldots, c_r$ be elements of $\mathrm{res}(L)$ that generate it over $\mathrm{res}(F)$; for every $i$, let $P_i$ be a polynomial in $i$ variables with coefficients in $\mathrm{res}(F)$ such that $P_i[c_1,\ldots, c_{i-1}, T]$ is the minimal polynomial of $c_i$ over $\mathrm{res}(F)[c_1,\ldots, c_{i-1}]$.
Choose a  lift $Q_i$ of $P_i$ monic in $T$ with coefficients in the ring of integers of $F$, and an element $R_i$ of $F(X)[Y]$ such that $R_i (a, b)$ is a lift of $c_i$. 
Let $\Phi(x,y)$ be the formula
\[\val (R_i(x,y))=0\;\text{and}\;\val [Q_i(R_1(x,y),\ldots, R_i(x,y))]>0\;\text{for all}\;i.\]

Now $L^{\mathrm h}$ is a compositum of $L$ and $K^{\mathrm h}$, so it is generated by $b$ over $K^{\mathrm h}$. 
Hence there exists a sub-tuple $\beta$ of $b$ of size $d$ such that $b$ is contained in the $K^{\mathrm h}$-vector space generated by
$\beta$. As $K^{\mathrm h}$ is the definable closure of $K$, the latter property can be rephrased as $\Psi(a,b)$ for some quantifier-free formula $\Psi$ in the language of valued fields, with parameters in $F$. 

Now let $L':=F(a',b')$ be a valued extension of $F$,
and set $K'=F(a')$. Assume that $\mathrm{res}(K')$ is
a regular extension of $\mathrm{res}(F)$, and that
\[L'\models \Phi(a',b')\;\text{and}\;\Psi(a',b').\] Then $\Psi(a',b')$
ensures that $(L')^{\mathrm h}$ is at most $d$-dimensional over $(K')^{\mathrm h}$,
while $\Phi(a',b')$ ensures that $\mathrm{res}(L')$ contains a field isomorphic to
$\mathrm{res}(F)(c_1,\ldots, c_r)=\mathrm{res}(L)$. Since $\mathrm{res}(K')$
is regular over $\mathrm{res}(K)=\mathrm{res}(F)$, the residue field $\mathrm{res}(L')$ contains a 
field isomorphic to $\mathrm{res}(L)\otimes_{\mathrm{res}(F)}\mathrm{res}(K')$, 
which is of degree $d$ over $\mathrm{res}(K')$. 
As a consequence, 
\[[L':K']=[\mathrm{res}(L'):\mathrm{res}(K')]=d\]
and thus
\[\val(L')=\val(K').\qedhere\]
\end{proof}

\subsection{Generic types of closed balls}\label{gentypeball}
In practice, the above proposition will be applied for $a$ realizing the generic type of a ball over $F$. Let us collect here some basic facts about such types. 
If $\gamma$ is an element of $\Gamma$, we denote by $r_\gamma$ the type of the closed ball of (valuative) radius $\gamma$, which belongs to
$\widehat{\mathbb A^1}$ and even to $\stda{\mathbb A^{1}}$. More generally if $\gamma=(\gamma_1,\ldots, \gamma_n)$ we shall denote
by $r_\gamma$ the type $r_{\gamma_1}\otimes \ldots\otimes r_{\gamma_n}$, which is the generic type of the $n$-dimensional ball of 
polyradius $(\gamma_1,\ldots, \gamma_n)$ and belongs to $\stda{\mathbb A^{n}}$. 

Now let $F$ be a  valued field, let $K$ be a
valued extension of $F$ and let $a_1,\ldots, a_r, a_{r+1},\ldots, a_n$ be elements of $K^\times$. 
Assume the following: 
\begin{enumerate}[1]
\item the group elements $\val(a_1), \ldots, \val(a_r)$ are $\mathbb Z$-linearly independent over $\val (F)$; 
\item one has $\val(a_i)=0$ for $i=r+1,\ldots, n$ and the residue classes of the $a_i$ for $i=r+1,\ldots, n$ are algebraically independent
over $\mathrm{res}(F)$. 
\end{enumerate}
Set $\gamma_i=\val(a_i)$ for $i=1,\ldots, n$. Then under these assumptions one has $a\models r_\gamma|_{F(\gamma)}$. 

Conversely,  assume that $a\models r_\gamma|_{F(\gamma)}$. Then $\val(a_i)=0$ for $i=r+1,\ldots, n$, the
residue classes of the $a_i$ for $i=r+1,\ldots, n$ are algebraically 
independent over $\mathrm{res}(F)$ and
$\mathrm{res}(F(a_{r+1},\ldots, a_n)$ is generated by the residue classes of the $a_i$, so is purely transcendental of
degree $n-r$ over $\mathrm{res}(F)$. 
In particular, this is a regular extension of $\mathrm{res}(F)$. 
Now the $\val(a_i)$ for $i=1,\ldots, r$ are $\mathbb Z$-linearly independent, the group $\val(F(a_1,\ldots, a_n))$ is generated over
the group $\val(F(a_{r+1},\ldots a_n))=\val(F)$ by the $\val(a_i) $ for $i=1,\ldots, r$, so it is free of rank $r$ modulo $\val(F)$; and the residue field of $F(a_1,\ldots, a_n)$
is equal to that of $F(a_{r+1},\ldots, a_n)$, so it is purely transcendental of degree $n-r$ over $\mathrm{res}(F)$; in particular, 
this is a regular extension of the latter. 

\begin{lemm}\label{lemm-abhyankar}
Let $F$ be a valued field and let $p$ be a strongly stably dominated (global) type with canonical parameter
of definition $\gamma\in \Gamma^n$ over $F$. Let $b\models p|_{F(\gamma)}$ and set $K=F(b)$. Then: 
\begin{enumerate}
\item $\gamma$ is definable over $F(b)$ ; 
\item $F(b)$ is an Abhyankar extension of $F$. 
\end{enumerate}
\end{lemm}
\begin{proof}
Let us start with (1). Let $\Phi$ be an automorphism of the monster model fixing $F(b)$ pointwise. One has to show that $\Phi$ fixes $\gamma$, or $p$ --
this amounts to the same. Set $\delta=\Phi(\gamma)$ and $q=\Phi(p)$. Let $A$ be a $\Phi$-invariant subset of $\Gamma$ containing $\gamma$. 
Since $p$ is orthogonal to $\Gamma$,
the restriction $p|_{F(\gamma)}$ implies a complete type $r$ over $F(A)$, which coincides necessarily with the type of $b$ over $F(A)$. Thus $p$ contains the type of $b$ over $F(A)$, and so does $\Phi(p)$ since both $b$ and $A$ are $\Phi$-invariant. So $p$ and $\Phi(p)$ are two global generically stable $F(A)$-definable types that coincide over $F(A)$; it follows that they are equal, cf.
Proposition 2.35 of \cite{simon}.

Now we prove (2). By replacing $\gamma$ by a suitable subtuple if necessary, we may and do assume that
$\gamma=(\gamma_1,\ldots, \gamma_n)$ where the $\gamma_i$ are $\mathbb Z$-linearly independent over $\val (F)$. 
Now choose $c=(c_1,\ldots, c_n)$ realizing $r_\gamma$ over $F(b)$. Then by stable domination, $b$ realizes $p$ over $F(\gamma, c)$
and in particular over $F(c)$. The type $p$ is strongly stably dominated, and it is definable over $F(c)$ by construction. So $\mathrm{res}(F(b))$
is of transcendence degree $\dim X$
over $\mathrm{res}(F(c))$, and $F(b,c)$ is thus Abhyankar over $F(c)$, hence over
$F$ since $F(c)$ is Abhyankar over $F$. Then $F(b)$ is Abhyankar over $F$.
\end{proof}

\begin{theo}\label{theo-tfinite-pure}Let $F$ be a defectless valued field with divisible value group. Let 
$X$ be  an $n$-dimensional integral $F$-scheme of finite type, and let  $\Upsilon$ be a $\Gamma$-internal subset 
 of $X^{\#}_{\mathrm{gen}}\subseteq \widehat X$.  
Then  $\SS(\Upsilon)$ 
is  finitely $(w,+)$-generated up to constant functions. 
\end{theo}

\begin{proof}
We shall prove the following: for every $p\in \Upsilon$, there exists a $F$-definable subset $W$ of $\Upsilon$
containing $p$ and finitely many functions $a_1,\ldots, a_n$
in $F(X)^\times$
such that for every $x\in W$ and every $f\in \SS(\Upsilon)$, the element $\val(f(x))$ of $\Gamma$ 
belongs to the group generated by $\val(F)$ and the $\val(a_i(x))$. This will allow us to conclude. 
Indeed, assume that this statement has been proved. Then by compactness there is a finite
cover $\mathscr W$ of $\Upsilon$
with finitely many sets $W$ as above. Hence
$\SS(\Upsilon)$ is $(w,+)$-generated by the $a_i$ up to 
constant functions.

Let $p\in \Upsilon$.  This is a strongly stably dominated global type. 
Let $\gamma \in \Gamma^r$ be a canonical parameter of definition 
of $p$ and let $b$ be a realization of $p$
over $F(\gamma)$. By
Lemma \ref{lemm-abhyankar} $\gamma$ is definable over $F(b)$ and $F(b)$
is Abhyankar over $F$. As $\gamma$ is definable over $F(b)$ and
as it is defined only up to inter-definability, we can assume that
$\gamma=(\gamma_1,\ldots, \gamma_r)$
where the $\gamma_i$ are $\mathbb Z$-linearly independent over
$\val(F)$, and where each $\gamma_i$ is equal to $\val(a_i)$
for some $a_i\in F(b)$. Since $p$ is stably dominated every element of
$\val(F(b))$ belongs to the $\mathbb Q$-vector space generated by $\val (F)$ 
and the $\gamma_i$. Moreover the group
$\val (F(b))$ is finitely generated over
$\val(F)$ because $F(b)$ is Abhyankar over $F$
and as $\val(F)$ is divisible, $\val( F(b))$ is torsion-free
modulo $\val (F)$; as  a consequence, $\val(F(b))/\val(F)$ is free of finite rank. We can thus
even assume that $(\gamma_1,\ldots, \gamma_r)$ is a $\mathbb Z$-basis
of $\val(F(b))/\val(F)$. 
The valued field $F(b)$ being
Abhyankar over $F$, the family $(a_1,\ldots, a_r)$ can be completed into
an Abhyankar basis $(a_1,\ldots, a_r, a_{r+1},\ldots, a_n)$ of $F(b)$ over $F$
such that  $\val(a_i)=0$ for every $i\geq r+1$ and the residue classes of $a_{r+1}$, \dots,  $a_n$ are
algebraically independent over the residue field of $F$.
The field $F(b)$ is then algebraic over $F(a_1,\ldots, a_n)$. 
We set $a=(a_1,\ldots,a_n)$ and we now denote by $\gamma$
the $n$-uple $(\gamma_1,\ldots, \gamma_n)$ with $\gamma_i=0$ if
$i\geq r+1$, so that $\gamma_i=\val(a_i)$
for all $i$.

Since $p$ is Zariski-generic, $a=(a_1,\ldots, a_n)$ can be interpreted as an $n$-uple of rational functions on $X$, giving rise to a map $\pi$ from a dense
open subset of $X$ to $\mathbb A^n_F$. In particular, $\pi$ induces a map (which we still denote by $\pi$) from $\Upsilon$
to $\widehat{\mathbb A^n_F}$, and the fact that $a_1,\ldots, a_n$
is an Abhyankar basis of $F(b)$ means that $\pi(p)|_{F(\gamma)}=r_{\gamma}|_{F(\gamma)}$; 
as both $\pi(p)$ and $r_\gamma$ are generically stable types defined over $F(\gamma)$, it follows that $\pi(p)=r_\gamma$. 

Moreover, the tower $F(a_{r+1},\ldots, a_n)\subseteq F(a)\subseteq F(a,b)$ fulfills
the conditions of Proposition \ref{prop-formula-valuegroup}; hence the latter provides a formula 
$\phi(y, x_1,\ldots, x_r)$ with coefficients in $F(a_{r+1},\ldots, a_n)$,
which we can see as the evaluation at $(a_{r+1},\ldots, a_n)$ of a formula
$\psi(y,x_1,\ldots, x_n)$. 

Now let $W$ be the subset of $\Upsilon$ defined as the set of types $q$  satisfying the following conditions, with $\delta_i:=
\val(a_i(q))$ 
\begin{enumerate}[a]
\item $\pi(q)=r_\delta$; 
\item $\delta_i=0$ for $r+1\leq i\leq n$; 
\item $\psi(b(q), a_1(q), \ldots, a_n(q))$ holds. 
\end{enumerate}

Then $W$ is an $F$-definable subset of $\Upsilon$ -- as far as condition (a) is concerned this is
by Lemma 8.2.9 in \cite{HL}, and it contains $p$.
Now let $q$ be a point of $W$.  Set $b'=b(q)$ and $a'=a(q)$, and $\gamma'_i=\val (a'_i)$ for all $i$. 
Conditions (a) and (b) ensure that $F(a')$ has a residue field which is regular over $\mathrm{res}(F(a'_{r+1},\ldots, a'_n))$. 
Indeed, up to applying an invertible monomial transformation to $(a'_1,\ldots, a'_r)$ and renormalizing, we can assume that there is some $s$
such that $\val(a'_1), \ldots, \val (a'_s)$
are free modulo $\val(F)$ and that $\val(a'_t)=0$ for $s+1\leq t\leq r$, in which case the result is obvious since the residue field we consider is then purely transcendental of degree $r-s$ over that of $F(a'_{r+1}, \ldots, a'_n)$.

Using the fact that $F(a'_{r+1},\ldots, a'_n)\simeq F(a_{r+1}, \ldots, a_n)$ as valued extensions of $F$
(with $a'_i$ corresponding to $a_i$) and 
the definition of $\psi$, we see that $\val(F(b',a'))=\val(F(a'))$. In other words, the value group of
$q$ is generated by the $a_i(q)$ and $\val(F)$, which ends the proof. 
\end{proof}

Our purpose is now to show how the results of section \ref{t-finite-pure}
extend quite straightforwardly, at least on affine charts, when $\Upsilon$ is not assumed to consist only of
Zariski-generic points. 

\subsection{A more general setting}
We still denote by $k$ a defectless valued field with divisible value group. Let $X$
be an \textit{affine} $k$-scheme of finite type, and let $\Upsilon$ be a $\Gamma$-internal subset 
of $X^{\#} \subseteq \widehat X$. Let $X_1,\ldots, X_m$ be the irreducible Zariski-closed subsets of $X$ whose generic point supports
an element of $\Upsilon$ (it follows from Corollary 10.4.6 of \cite{HL} and finiteness of the Zariski topology of
$\Gamma_\infty^w$ that there is only a finite number of such irreducible subsets); for each $i$, set
\[X'_i=X_i\setminus \bigcup_{j, X_j\subsetneq X_i}X_j\] and $\Upsilon_i=\Upsilon\cap \widehat{X'_i}$. By construction, $\Upsilon=\coprod \Upsilon_i$ and   for 	all $i$, $\Upsilon_i$ consists only in 
Zariski-generic 
points in 
in $\widehat{X'_i}$. 
We denote by $\SS(\Upsilon)$ the set of $k$-definable functions of the form $\val (f)$ with $f$ a \emph{regular} function on $X$  (and not merely a rational function as above). 

\begin{prop}
There exists a finite set $E$ of regular functions on $X$ such that for every $f\in \SS(\Upsilon)$, there exists a finite covering $(D_a)_a$ of $\Upsilon$
by closed definable subsets and, for each $a$, an element $\lambda$ of $k$, a finite family $(e_1,\ldots, e_\ell)$
of elements of $E$, and a finite family $(\epsilon_1,\ldots, \epsilon_\ell)$ of elements of $\{-1,1\}$ 
such that: 
\begin{itemize}
\item[$\diamond$] $\epsilon_j=1$ if $e_j$ vanishes on $D_a$ ; 
\item[$\diamond$] $f=\val(\lambda e_1^{\epsilon_1}\ldots e_\ell^{\epsilon_\ell})$ identically on $D_a$. 
\end{itemize}
\end{prop}

\begin{proof}
For all $i$, we can apply Theorem \ref{theo-tfinite-pure}
to the integral scheme $X'_i$ and the $\Gamma$-internal set  $\Upsilon_i$; let $E_i$ be the finite set of rational functions on $X'_i$ provided
by this theorem. Write $E_i=\{g_{ij}/h_{ij}\}_j$ where $g_{ij}$ and $h_{ij}$ are non-zero regular functions
on the integral affine scheme $X_i$. 
For all $(i,j)$, let $g'_{ij}$ and $h'_{ij}$ denote lifts of $g_{ij}$ and $h_{ij}$ to the ring $\mathscr O_X(X)$. We then might take for $E$ the set of all
$g'_{ij}$ and $h'_{ij}$. 
\end{proof}

\section{Specialisations and Lipschitz embeddings}\label{section5}
As before, $\Upsilon$ is a $\Gamma$-internal subset of $X^{\#}_{\mathrm{gen}}$ for $X$ a separated integral scheme of finite type over a valued field $K$. The goal of this section is to show the existence of regular embeddings of $\Upsilon$ in some $\Gamma^n$ such that $\SS(\Upsilon)$ becomes exactly the set of Lipschitz definable functions under certain assumptions. We begin with some definitions.
\begin{defi}\label{def:embed}
Let $\alpha:\Upsilon\to \Gamma^n$ be a definable and continuous map and set $W=\alpha(\Upsilon)$.
\begin{enumerate}
    \item We say $\alpha$ is \emph{regular} if $\alpha$ is given by a tuple of regular functions $\Upsilon \to \Gamma$, \ie,  functions of the form $\val (f)$ with $f$ a non-zero
    rational function. 
    \item If $\alpha$ is an embedding, then we say $\alpha^{-1}:W\to \Upsilon$ is an \emph{integral parameterization} if for any rational function $f$ defined on $\Upsilon$, $\val(f)\circ \alpha^{-1}$ is piecewise $\mathbb{Z}$-affine. 
    We will also call $\alpha$ integral in this case.
    \item If $\alpha$ is an embedding, then we say $\alpha^{-1}:W\to \Upsilon$ is \emph{Lipschitz} if for any non zero rational function $f$ on $X$, $\val (f)\circ \alpha^{-1}$ is a Lipschitz function. We will also say  $\alpha$ is Lipschitz.
    \item We say $\alpha$ is a \emph{good embedding} if it is integral and Lipschitz. 
   \end{enumerate}
\end{defi}
It is immediate from the definition that if $\alpha: \Upsilon\to \Gamma^n$ is a regular embedding
(\resp regular integral embedding, \resp regular Lipschitz embedding)
and $f$ is another regular function $\Upsilon \to \Gamma$, then $(\alpha, \val (f))\colon\Upsilon\to\Gamma^{n+1}$ is also a regular embedding
(\resp regular integral embedding, \resp regular Lipschitz embedding).

\begin{lemm}\label{lemm-tfinite}
Assume that $K$ is algebraically closed
and let $\Upsilon$ be a $\Gamma$-internal subset of $X^{\#}_{\mathrm{gen}}$.
Then there exists a 
 regular integral embedding
$\alpha\colon  \Upsilon\hookrightarrow \Gamma^n$.
\end{lemm}

\begin{proof}By Theorem \ref{theo-tfinite-pure} there exists
a finite family $\alpha = (\alpha_1, \ldots, \alpha_n)$  which $(w,+)$-generates $\SS(\Upsilon)$ modulo the constant functions. Since $K$ is algebraically closed, it follows from \cite[Proposition 6.2.7]{HL}
that there exists a
regular embedding of $\Upsilon$.
We may thus enlarge $\alpha$ so that it becomes a regular embedding; it is integral by $(w,+)$-generation.
\end{proof}

We will now recall some basic facts about $\acvsf$ and specialisations, that will provide an important criterion for the existence of good embeddings.

\subsection{$\acvsf$-specialisations}\label{sec:specialisation}
We consider a triple $(K_2,K_1,K_0)$ of fields with surjective places $r_{ij} : K_i \to K_j$ for $i > j$, 
with $r_{20} = r_{10} \circ r_{21}$, 
such structures are also called $\vsf$. The places $r_{21}$ and $r_{20}$ give rise to two  valuations on $K_2$, which we denote by  $\val_{21}$ and $\val_{20}$ respectively. 
We denote by $\Gamma_{ij}$ and $\RES_{ij}$ the corresponding value groups and residue fields.
We consider $(K_2,K_1,K_0)$ as a substructure of a model of the theory $\acvsf$  introduced in~\cite[Chapter 9.3]{HL}. We will use $K_{210}$ to denote the structure $(K_2,K_1,K_0)$. It is clearly an expansion of $(K_2, \val_{21})$ via an expansion of the residue field and an expansion of $(K_2,\val_{20})$ by a convex subgroup in the value group. We will focus on the latter expansion.

Let $X$ be an affine integral scheme of finite type over $K_2$, we will use $X_{20}$ when we view $X$ as a definable set in an ambient model of $\acvf$ extending $(K_2,\val_{20})$ and $X_{21}$ is defined analogously. There is a natural map $s: X^{\#}_{20}\to X^{\#}_{21}$ which can be described as follows. 
Let $p \in X^{\#}_{20}$. By~\cite[Lemma 9.3.8]{HL}, we have that $p$ generates a complete type $p_{210}$ in $\acvsf$. 
Furthermore, by~\cite[Lemma 9.3.10]{HL}, $p_{210}$ as an $\acvsf$-type is stably dominated. Let $\dim(p)$ denote the dimension of the Zariski closure of $p$. Let $L\models\acvsf$ extending $K_{210}$ and $c\models p|L$. Since $p$ corresponds to an Abhyankar point in the space of valuations, we see that the residual transcendence degree of $\tp_{21}(c/L)$ is still $\dim(p)$, so $\tp_{21}(c/L)$ extends to a type $s(p)$ in $X^{\#}_{21}$. (Note that here we work in the restricted language where the only valuation is $\val_{21}$.)
\begin{lemm}
Let $Y\subseteq X^{\#}_{20}$ be an $\acvsf_{K_{210}}$-definable set, then $s|_Y$ is a definable function. 
\end{lemm}
\begin{proof}
By the way $s$ is defined, it is a pro-definable function by considering the $\varphi$-definitions. Note that a pro-definable function between two definable sets is definable by compactness. 
\end{proof}
We need one last lemma before stating our criterion with respect to specialisations.
\begin{lemm}\label{lem:imaginary}
Let $(K_2,K_1,K_0)\models \acvsf$ and $Y$ be a definable set of imaginaries in $\acvf_{K_{21}}$. If $Y$ is $\Gamma_{20}$-internal as a definable set in $\acvsf_{K_{210}}$, then $Y$ is $\Gamma$-internal in $K_{21}$.
\end{lemm}
\begin{proof}
By the classification of imaginaries in $\acvf$~\cite[Theorem 1.01]{haskell-h-m2006}, if $Y$ is not $\Gamma$-internal in $K_{21}$, there is an $\acvf_{K_{21}}$-definable map (possibly after expanding the language by some constants) that is generically surjective onto the residue field. By assumption, $Y$ is $\Gamma_{20}$-internal as an $\acvsf_{K_{210}}$ set. This yields a generically surjective map $\Gamma_{20}\to \RES_{21}$. Composing with the dominant place $\RES_{21}\to \RES_{20}$, we obtain an $\acvsf$-definable map $\Gamma_{20}\to \RES_{20}$ that is generically surjective. By~\cite[Lemma 9.3.1(4)]{HL}, one checks immediately that the two sorts $\Gamma_{20}$ and $\RES_{20}$ are orthogonal in $\acvsf$, hence a contradiction.
\end{proof}

\subsection{Specialisable maps and Lipschitz condition} 
Now we introduce the notion of specialisations of maps. Let $(K,v)$ be a valued field, we denote by  $\rho(K)$  the set of convex subgroups of $\Gamma(K)$. Clearly, if $K$ is of transcendence degree $m$ over the prime field, then $|\rho(K)|\leq m+1$. For each $\Delta\in \rho(K)$, we have a valuation $\val_{21}:K\to \Gamma(K)/\Delta$ given by quotienting out by $\Delta$, which gives rise to a $\vsf$ structure we shall denote by  $K[\Delta]$. Each choice
of  $\Delta$ specifies an expansion of $\acvf_K$ to $\acvsf_{K}$ by interpreting the convex subgroup to be the convex hull of $\Delta$. Moreover, by varying $\Delta$ one exhausts all the possible expansions of $\acvf_K$ to $\acvsf_{K}$. Let $X$ be an integral separated scheme of finite type over $K$ as before. We write $X_\Delta$ to denote $X$ as a definable set in $\acvf_{K[\Delta]}$. We use $s_\Delta$ to denote the map $s$ defined in Section~\ref{sec:specialisation} when we expand $\acvf_K$ to $\acvsf_{K[\Delta]}$.  We use $\Upsilon_\Delta$ to denote $s_\Delta(\Upsilon)$. Similarly, if $\alpha:\Upsilon\to \Gamma^n$ is some regular embedding, we use $\alpha_\Delta:\Upsilon_\Delta\to \Gamma_{21}^n$ to denote the corresponding map. 
\begin{defi}
Let $\alpha:\Upsilon \hookrightarrow  \Gamma^n$ be a regular embedding and
let $K$ be a field over which
$\alpha$ is defined. We say $\alpha$ is \emph{specialisable} if
for every convex subgroup $\Delta$ of $\Gamma(K)$ the map $\alpha_\Delta$ is still an embedding. 
\end{defi}

\begin{rema}\label{rem:parameter}
Note that the specialisability of $\alpha$ does not depend on the choice of $K$. Namely, let $L\supseteq K$ be an extension of valued fields, it suffices to show that if $\alpha$ is specialisable with respect to $K$, it is so with respect to $L$. Let $\Delta_L$ be a convex subgroup of $\Gamma(L)$. Note that this gives a convex subgroup $\Delta_K$ of $\Gamma(K)$ by taking intersection. Note that whether $\alpha_{\Delta_L}$ is an embedding only depends on $\acvf_{L[\Delta_L]}$, which is an expansion of $\acvf_{K[\Delta_K]}$. Hence the specialisability of $\alpha$ over $K$ guarantees that $\alpha_{\Delta_L}$ is an embedding.
\end{rema}

\begin{rema}
If $\alpha\colon \Upsilon \hookrightarrow  \Gamma^n$ is a specialisable
regular
embedding and $\beta\colon \Upsilon \to \Gamma^m$ is
any regular map, the
regular embedding $(\alpha,\beta)\colon
\Upsilon\to \Gamma^{n+m}$ is specialisable as well. 

\end{rema}

\begin{rema}\label{rem:specialisable}
Assume  $\alpha:\Upsilon\to \Gamma^n$ is specialisable and defined over $K$, and $\Upsilon'\subseteq \Upsilon$ is  definable but not necessarily over $K$. If $\alpha$ is specialisable, so is $\alpha|_{\Upsilon'}$.
This follows from a similar argument as in Remark~\ref{rem:parameter}.
More precisely, let $L\supseteq K$ be such that $\Upsilon'$ is defined over $L$, any expansion of $\acvf_L$ to $\acvsf_L$ by some $\Delta'$ gives an expansion of $\acvf_K$ to $\acvsf_K$ by some $\Delta$. As $\alpha$ is specialisable, $\alpha_\Delta$ is an embedding for any $\Delta$, thus $\alpha|_{\Upsilon'}$ is specialisable.
\end{rema}

\begin{theo}\label{thm:existence}
Let $X$ be an affine integral scheme of finite type over a valued field $K$ and let $\Upsilon\subseteq X^{\#}$ be a $\Gamma$-internal subset. Let $F$ be
a finitely generated field over which all the above is defined.
Then there exists a $F^{alg}$-definable integral regular embedding of $\Upsilon$ into $\Gamma^n$ that is specialisable.
\end{theo}
\begin{proof}
For each $\Delta\in \rho(F)$, by Lemma~\ref{lem:imaginary}, we have that $\Upsilon_\Delta\subseteq X^{\#}_\Delta$ is $\Gamma$-internal in $\acvf_{F[\Delta]}$.

Consider $X$ as embedded in some affine space. 
By~\cite[Corollary 6.2.5]{HL}, for each $\Delta$, there are finitely many polynomial functions $h^\Delta_i $ such that $h^\Delta=(\val(h_1^\Delta),\ldots,\val(h_s^\Delta))$ is injective on $\Upsilon$.

Moreover the $h^\Delta_i$'s can be found to be defined over $F^{alg}$ by the
proof of~\cite[Corollary 6.2.5]{HL} (or more precisely,~\cite[Lemma 6.2.2]{HL}). Since there are only finitely many such $\Delta$'s to consider,
putting them as the coordinates,
we get some specialisable embedding as desired,
which
can be made integral by concatenation with an arbitrary integral regular embedding, whose existence
follows from Lemma \ref{lemm-tfinite}. 
\end{proof}

\begin{rema}
In the situation of interest for classical non-archimedean geometry, the ground field $K$ will be algebraically closed and equipped
with a valuation whose group embedds into $\mathbb{R}$ and has therefore no non-trivial proper convex subgroup. The reasoning above then shows 
that \textit{any} $K$-definable regular embedding from $\Upsilon$ into $\Gamma^s$ is specialisable.
\end{rema}

\begin{prop}\label{speandlip}
Let $X$ be an affine integral scheme of finite type over a valued field $K$ and let $\Upsilon\subseteq X^{\#}$ be a $\Gamma$-internal subset. If $\alpha:\Upsilon\hookrightarrow  \Gamma^n$ is a specialisable embedding, then the image of $\mathbb{S}(\Upsilon)$ is contained in the group of Lipschitz functions. In other words, all the log-rational functions are Lipschitz and $\alpha$ is Lipschitz.
\end{prop}
\begin{proof}
We let $W=\alpha(\Upsilon)$ and use $p_w$ to denote $\alpha^{-1}(w)$ for $w\in W$. Assume there is some $f\in K(X)$ such that $w\mapsto p_w(f)$ is not Lipschitz. Going to an elementary extension, we may assume there is $w_1,w_2 \in W$ such that $|p_{w_2}(f)-p_{w_1}(f)|>n|w_1-w_2|$ for all $n\in \mathbb{N}$. Take $C$ to be the convex subgroup generated by $|w_1-w_2|$. Consider $L$ to be the same field with the valuation given by quotienting out by $C$. By our assumption on specialisability, we have that $\alpha_L$ is an embedding. However, we have
$\overline{w_1}=\overline{w_2}$, while $\overline{p_{w_1}(f)}=p_{w_1}(f)+C\neq p_{w_2}(f)+C=\overline{p_{w_2}(f)}$, a contradiction.
\end{proof}

\begin{coro}\label{goodexistence}
Let $X$ be an affine integral scheme of finite type over an algebraically closed valued field $K$ and let $\Upsilon\subseteq X^{\#}$ be a $\Gamma$-internal set.
Then there exists a good embedding $\Upsilon \hookrightarrow \Gamma^n$.
\end{coro}

\begin{proof}
The embedding provided by 
Theorem \ref{thm:existence}
is $K$-definable, and it is good in view of
Proposition \ref{speandlip}.
\end{proof}

\section{The main theorem}
In this section, we prove the theorem stated in Section~\ref{sec:about-proof} and we transfer it into the Berkovich setting. 

\begin{lemm}\label{lem-stable-min}
Let $k$ be a valued field with infinite residue field, let $X$ be a geometrically integral $k$-scheme and let
$\Upsilon\subset X^\#
_{\mathrm{gen}}$ be a $k$-definable $\Gamma$-internal subset defined
over $k$.  The group $\SS(\Upsilon)$ is stable under $\min$ and $\max$.
\end{lemm}

\begin{proof}
It is enough to prove stability under $\min$. Let $p$ be a point of $\Upsilon$. If there exists a scalar $a$ of valuation zero such that
$\val(f(p)+ag(p))>\min(\val (f)(p), \val (g(p)))$ then $\mathrm{res}(a)$ is a well-defined element of the residue field
which we call $\theta(p)$; otherwise we set (say) $\theta(p)=0$. 
Then $\theta$ is a $k$-definable map from the $\Gamma$-internal set $\Upsilon$ to the residue field. 
By orthogonality between the value group and the residue sorts, $\theta$ has finite image. Since $k$ has infinite residue field, 
there exists an element $a\in\mathscr O_k^\times$ whose residue class does not belong to the image of $\theta$. 
Then $f+ag\neq 0$ and $\val(f(p)+ag(p))=\min (\val (f(p)), \val (g(p)))$ for all $p\in \Upsilon$. 
\end{proof}

In the situation of the lemma above, it thus makes sense to speak about
an $(\ell,+)$-generating system of elements of $\SS(\Upsilon)$. As for $(w,+)$-generation, we shall say for short that $\SS(\Upsilon)$
is finitely $(\ell,+)$-generated up to the constant functions if 
there exists a finite subset $E$ of $\SS(\Upsilon)$ such that $E$ and the $k$-definable constant functions (\ie, the constant functions taking values in $\mathbb{Q}\otimes \val(k^\times)$) $(\ell,+)$-generate $\SS(\Upsilon)$.

\begin{theo}\label{themaintheorem}
Let $k$ be
an algebraically closed valued field. Let $X$ be an integral scheme of finite type over $k$  and let $\Upsilon\subseteq X^\#_{\mathrm{gen}}$ be a $\Gamma$-internal
subset defined over $k$.
The group $\SS(\Upsilon)$
is stable under $\min$ and $\max$ and
is finitely $(\ell,+)$-generated up to constant functions. 
\end{theo}

\begin{proof}
By Theorem~\ref{thm:existence}, there is a $k$-definable
good embedding $\alpha: \Upsilon\to \Gamma^n$ for some $n$. 
By Theorem~\ref{theo-tfinite-pure}, $\SS(\Upsilon)$ is
$(w,+)$-finitely generated up to  constant functions.
Let $f_1,\ldots,f_m$ be finitely
many $k$-rational functions whose valuations $(w,+)$-generate $\SS(\Upsilon)$ up to 
constant functions, adjoining the $\val (f_i)$ as new coordinates of
$\alpha$, we may furthermore assume
that $\SS(\Upsilon)$ 
is $(w,+)$-generated by the components of $\alpha$
and the constant functions. 
By possibly enlarging once again $\alpha$ and replacing
$X$ with a suitable dense Zariski-open subset we can also assume that
$\alpha=\val (f)$ for some closed immersion
$f\colon X\to \gm^n$;
in particular, $\alpha$ is definably proper
and induces a definable
homeomorphism $\Upsilon\simeq \alpha(\Upsilon)$.

Let $f$ in $\SS(\Upsilon)$. Since $\alpha$ is a specialisable embedding whose coordinates $(w,+)$-generate
$\SS(\Upsilon)$ up to the constant functions, the composition
$f\circ \alpha^{-1}$ viewed
as a $\Gamma$-valued function on 
$\alpha(\Upsilon)$ is piecewise $\mathbb{Z}$-affine and Lipschitz. In view of
Theorem~\ref{thm:lip-t-gen}, this implies that
$f\circ \alpha^{-1}$ is an $\ell$-combination of finitely
many $\mathbb{Z}$-affine functions,
so that $f$ itself
is an $(\ell,+)$-combination of the
components of $\alpha$ and of 
constant functions.
\end{proof}

\begin{rema}\label{wgeneration-scalarextension}
Assume that $k$ is algebraically closed and let
$(f_1,\ldots,f_n)$ be a family of rational functions on $X$ such that 
$\SS(\Upsilon)$ is $(w,+)$-generated
(resp. $(\ell,+)$-generated)
by the $\val (f_i)$ and the constant
($k$-definable) functions.
Then for every algebraically closed
extension $L$ of $k$,
the $\val (f_i)$ and the $L$-definable
constant functions $(w,+)$-generate (resp. 
$(\ell,+)$-generate) $\SS_L(\Upsilon)$
(work with a bounded family of
rational functions and use compactness). 
\end{rema}

Our purpose is now
to state and prove the Berkovich avatar of our main theorem. We fix a non-archimedean complete field $F$. For all $n$, we denote by $S_{F,n}$ the closed subset $\{\eta_r\}_{r\in (\mathbb{R}^\times_+)^n}$
of $\mathbf G_{\mathrm m,F}^{n,\mathrm{an}}$, where $\eta_r$ is the semi-norm
$\sum a_I T \mapsto \max \abs{a_I}r^I$.

In \cite{confluentes}, 4.6  a general notion of a skeleton is defined for an
$F$-analytic space; the subset $S_{n,F}$ 
of $\mathbf G_{\mathrm m,F}^{n,\mathrm{an}}$
is the archetypal example of such an object. 

But this notion is however slightly too analytic for our
purposes here: indeed, if $X$ is an algebraic variety
over $F$ then $X\an$ might have plenty of skeleta
in the sense of \cite{confluentes} that cannot be handled by our methods, since they would not
correspond to any $\Gamma$-internal subset  of $\widehat X$, by lack of algebraic definability. 
For instance, assume that $F$ is algebraically closed and non-trivially valued, and let $f$ be any non-zero analytic function of $\mathbf A^{1,\mathrm{an}}_F$ with countably many zeroes. Let $U$ be the non-vanishing locus of $f$, and let $\Sigma$ be the preimage of $S_{1,F}$ under $f
\colon U\to \gm$. Then $\Sigma$ is a skeleton in the sense of \cite{confluentes}, but topologically this is only a locally finite graph, with countably many branch points. 
We thus shall need to focus on ``algebraic" skeleta. 

\begin{theo}\label{mtberk}
Let us assume that $F$ is algebraically closed. Let $X$ be an integral $F$-scheme of finite type, and let $n$ be its dimension. Let $\phi_1,\ldots,\phi_r$ be maps $U_i\to \mathbf G_{\mathrm m, F}^n$ where the $U_i$ are dense open subsets of $X$, and let $S\subseteq X\an$ be a subset of
$\bigcup_i \phi_i^{-1}(S_n)$ defined by a Boolean combination of norm inequalities between non-zero rational functions. 

There exist finitely many non-zero rational functions $f_1,\ldots,  f_m$ on $X$ 
such that the following hold. 

\begin{enumerate}[1]
\item
The functions $\log \abs{f_1},\ldots, \log \abs{f_m}$ identify $S$
with a piecewise-linear subset of $\mathbb{R}^m$ (\ie, a subset defined by a Boolean combination
of inequalities between $\mathbb{Q}$-affine functions). 
\item The group of real-valued functions on $S$ of the form $\log \abs g$ for
$g$ a non-zero rational function on $X$
is stable under $\min$ and $\max$ and is 
$(\ell,+)$-generated by the $\log \abs{f_i}$
and the constant functions of the form $\log \abs{\lambda}$ with $\lambda \in F^\times$. 
\end{enumerate}
\end{theo}

\begin{proof}
The subset $\Sigma$ of $\widehat X$ given by the same definition as $S$ mutatis mutandis is a $\Gamma$-internal set
containd in $X^\#_{\mathrm{gen}}$ to which we can thus apply 
Theorem \ref{themaintheorem}. The theorem above then follows by noticing that if $L$ denotes a non-archimedean maximally complete extension of $F$ with value group the whole of $\R_+^\times$, then $S$ is naturally homeomorphic to $\Sigma (L)$.
\end{proof}

\begin{rema}Note that by Theorem
\ref{structure-pure-gammainternal}
the condition that $S$ is a subset of some
$\bigcup_i \phi_i^{-1}(S_n)$
holds as soon as  $S$ is the image of  $\Upsilon (L)$ under the projection $\widehat{X} (L) \to X\an$ 
with $\Upsilon$ some $F$-definable
$\Gamma$-internal  subset of $X^{\#}_{\mathrm{gen}}$ and $L$ as in the above proof.
\end{rema}

\begin{rema}\label{contrex}
We insist that we require that the ground field
be algebraically closed. Indeed, our theorems (for stable completions
as well as for Berkovich spaces)
definitely
do not hold over an arbitrary non-Archimedean field, 
\textit{even in a weaker version with $(w,+)$-generation 
instead of $(\ell,+)$-generation}, as witnessed
by a counter-example that was communicated to the authors
by Michael Temkin (this counterexample involves a field with defect, we do not know if our theorem holds
for defectless fields with divisible value group as in Theorem \ref{theo-tfinite-pure}).

For the reader's convenience we will first detail
the original counter-example
which is written
in the Berkovich's language,
and then a model-theoretic variant
thereof in the 
Hrushovski-Loeser's language. 

\subsubsection{The Berkovich version}
Let $F$ be a non-archimedean field and let $\mathbb F$
be the completion of an algebraic closure of $F$ ; assume that
the residue field of $F$
is of positive characteristic $p$ and that
$F$ admits an immediate extension $L$ of degree $p$, say $L:=F(\alpha)$ 
with $\alpha\in \mathbb F$. By general valuation theory, the distance $r$ between $\alpha$ and $F$ is not achieved. 

For every $s\geq r$ let $\xi_s$ be the image on
$\mathbf P^{1,\mathrm{an}}_F$ of the
Shilov point of the
closed $\mathbb F$-disc
with center $\alpha$ and radius $s$. If $s>r$ there exists $\beta_s$
in $F$ with $\abs{\alpha-\beta_s}\leq s$,
and $\xi_s$ is the Shilov point of the closed $F$-disc
with center $\beta_s$ and radius $s$; but as far as $\xi_r$ is concerned,
it is the Shilov point of an affinoid domain $V$ of $\mathbb P^{1,
\mathrm{an}}_F$ without rational point. 

Let $v$ be a rigid point of $V$. It corresponds to an element $\omega$ 
of $\mathbb F$ algebraic over $F$ and whose distance to $F$ is equal to $r$ and not achieved. Therefore the extension $F(\omega)$ has defect over $F$, which forces its degree to be divisible by $p$. In other words, $[\mathscr H(v):F]$ is divisible by $p$. 

In particular if $f$ is any non-zero element
of $F(T)$, the divisor
of $f|_V$ has degree divisible by $p$, so that there
exists some $s>r$ such that the slope of $\log \abs f$ on $(\xi_r,\xi_s)$ is divisible by $p$. 

Now assume that there exists a finite set
$E$ of non-zero rational functions such
that on the skeleton $[\xi_r,\infty)$, every function of the form
$\log \abs g$ with $g$
in $F(T)^\times$ belongs piecewise to the group generated by 
the $\log \abs h$ for $h\in E\cup F^\times$.
Then there would exist some $s>r$
such
that for every $g$ as above,
all slopes of $(\log \abs g)|_{[\xi_r,\xi_s]}$
are divisible by $p$.
Taking $g=T-\beta_s$
leads to a contradiction.

\subsubsection{The model-theoretic version}
Let $F$ be a perfect valued field of
positive residue characteristic $p$
such that there exists an irreducible
separable
polynomial $P\in F[T]$ with the following
property: \textit{the smallest closed
ball containing
all roots of $P$ has no $F$-rational points, but any
bigger $F$-definable closed ball has one $F$-point}
(it is not difficult to exhibit such pairs
$(F,P)$; the easiest case
is that of pure characteristic $p$, where one can take
any perfect field $F$ with an height 1 valuation 
having an element $s$ with $\val(s)<0$
such that $T^p-T-s$ has no root in $F$, and
take $P=:T^p-T-s$; for instance, the perfect closure
of $\F_p(s)$ equipped
with the $(1/s)$-adic valuation will do the job). 

Let $b$ be the smallest closed ball 
containing the roots of $P$, and let 
$B$ be a bigger $F$-definable
closed ball. Let $I$ be the interval between their generic points; this is
a
$\Gamma$-internal subset of $\widehat{\P^1}$
contained in $\P^{1,\#}_{\mathrm{gen}}$. 
This interval is naturally parameterized by the interval $[V,v]$ where $V$
is the valuative radius of $B$ and $v$ is that of $b$, 
and we will identify them. In particular a linear function
from $I$ to $\Gamma$
has a well-defined slope.  We will be interested in the germ of functions on $I$ towards
the endpoint $v$. 
The number of roots in $b$ of every 
irreducible polynomial of $F[T]$ is divisible
by $p$, 
for otherwise averaging the roots would produce an $F$-rational
point in $b$. Hence the valuation of every polynomial,
and indeed every rational function in $F(T)$, has slope divisible by $p$
on some interval $(i,v)$
inside $I$. If the group of functions $\val(f) | I$
were finitely $(w,+)$-generated 
up to constants, there would be a fixed $i<v$ (defined over $F$) such that all
val-rational functions have slope divisible by $p$ on $[i,v]$. Now pick an 
$F$-rational point $a$ in the closed ball containing $b$ of valuative
radius $(i+v)/2$; then $T-a$ has slope one on $(i, (i+v)/2)$, contradiction.  
\end{rema}

\section{Applications to (motivic) volumes of skeleta}\label{sec:volumes}

It follows from Theorem \ref{theo-tfinite-pure} on finite $(w,+)$-generation that skeleta are endowed with a canonical piecewise
$\mathbb{Z}$-affine structure. In
this section we explain how 
this implies the existence of canonical volumes for skeleta.

\subsection{Some Grothendieck semirings of $\Gamma$}
We shall consider various Grothendieck  semirings of $\Gamma$
analogous to those introduced in \S 9 of \cite{HK} (see also \cite{HK2} for a detailed study of the rich structure of such semirings).
Let $\Gamma$ be a non-trivial
divisible ordered abelian group 
and let $A$ be a fixed subgroup of $\Gamma$.
We work in 
the theory $\DOAG_A$ of (non-trivial)
divisible ordered Abelian groups with distinguished constants for elements of the subgroup  $A$.
Fix a non negative integer $N$. One defines a category $\Gamma (N)$ as follows (since there is no risk of confusion we omit the $A$ from the notation).
An object of $\Gamma (N)$ is a finite disjoint union of subsets of $\Gamma^N$ defined by linear equalities and inequalities with $\mathbb{Z}$-coefficients and constants in $A$. A morphism $f$ between 
two objects 
$X$ and $Y$  of $\Gamma (N)$ is a bijection such that  there exists a finite partition $X = \bigcup_{1 \leq i \leq r} X_i$ with
$X_i$ in $\Gamma (N)$, matrices 
$M_i \in \mathrm{GL}_N(\mathbb{Z})$ and constants $a_i \in  A^N$, such that for $x \in X_i$,
$f (x) = M_i x + a_i$.
We denote by $K_+ (\Gamma (N))$  the Grothendieck semigroup of  this category, that is the free abelian semigroup generated by isomorphism
classes of objects of  $\Gamma (N)$ modulo the cut and paste relation $[X] = [X\setminus Y ] + [Y]$ if $Y \subseteq X$.
The inclusion map $\Gamma^N \to \Gamma^{N+1}$ given by $x \mapsto (x, 0)$ induces an inclusion functor
$\Gamma (N) \to \Gamma (N+1)$ and we denote by $\Gamma (\infty)$ the colimit of the categories 
$\Gamma (N)$. We may identify the Grothendieck semigroup $K_+ (\Gamma (\infty))$ of $\Gamma (\infty)$ with the colimit of the semigroups
$K_+ (\Gamma (N))$. It is endowed with a natural structure of a semiring.
We may also consider the  full subcategory   $\Gamma^{\mathrm{bdd}} (N)$ of
$\Gamma (N)$
consisting of bounded sets, that is definable subsets of $[-\gamma, \gamma]^N$ for some non negative $\gamma \in \Gamma$ (which can be chosen in $A$),
and the  corresponding full subcategory   $\Gamma^{\mathrm{bdd}} (\infty)$ of
$\Gamma (\infty)$ and its Grothendieck semiring.
The above categories admit   natural filtrations $F^{\cdot}$ by dimension, with 
$F^n$
the subcategory generated by objects of o-minimal dimension $\leq n$ and we will also consider the induced filtration on Grothendieck rings.

\subsection{Volumes}
Let $R$ be a real closed field. Fix integers $0 \leq n \leq N$.
Let $W$ be a  bounded piecewise $\mathbb{Z}$-linear definable subset of  $R^N$ of o-minimal dimension $n$. We denote by 
$\mathrm{vol}_{n} (W)$ its 
$n$-dimensional volume which can be defined in the following way. After decompositing into simplices, it is enough to define the volume
of a simplex spanned by $n +1$-points, which one can do via the classical formula over $\mathbb{R}$, choosing
the normalization such that,
for any family $(e_1, \dots, e_n)$  of $n$ vectors in $R^N$ with integer coordinates which can be completed to 
a basis of the abelian group $\mathbb{Z}^N$, the volume of the simplex with vertices the origin and the endpoints of $e_1$, \dots, $e_n$ is $\frac{1}{n!}$.
When $R = \mathbb{R}$, $\mathrm{vol}_{n}$ is well-defined thanks to the existence of the Lebesgue measure. In general, 
the well-definedness of  $\mathrm{vol}_{n}$ 
follows from the case of $\mathbb{R}$ since
after increasing $R$ one can assume it 
is an 
an elementary extension of $\mathbb{R}$.

Thus, for any embedding $\beta: A \to R$ with $R$ a real closed field
and any integer $n$,
$\mathrm{vol}_{n}$
 induces a morphism
 $\mathrm{vol}_{n, \beta} : F^nK^{\mathrm{bdd}} _+ (\Gamma (N))
 /F^{n-1} K^{\mathrm{bdd}} _+ (\Gamma (N))\to R$ which stabilizes to a morphism
 $\mathrm{vol}_{n, \beta}  : F^nK^{\mathrm{bdd}} _+ (\Gamma (\infty))/F^{n-1} K^{\mathrm{bdd}} _+ 
 (\Gamma (\infty)) \to R$.

\subsection{Motivic volumes of skeleta}\label{mvs}
Let us assume that we are in the setting of Theorem \ref{theo-tfinite-pure},
that is,  $k$ is a defectless valued field with divisible value group,
$X$ is  an $n$-dimensional integral $k$-scheme of finite type and   $\Upsilon$ is a  $\Gamma$-internal subset 
 of $X^{\#}_{\mathrm{gen}}\subseteq \widehat X$.  
Then, by
Theorem \ref{theo-tfinite-pure},
$\SS(\Upsilon)$ 
is  finitely $(w,+)$-generated up to constant functions. 
Let $\alpha : 
\Upsilon \to \Gamma^N$
be a definable embedding of the form
$(\val (f_1), \cdots, \val(f_N))$ where the functions 
$\val (f_i)$ are 
$(w,+)$-generating $\SS(\Upsilon)$ up to constant functions.
We take for $A$ the group $\val (k^{\times})$.

\begin{prop}The class of $\alpha (\Upsilon)$ in
$K_+ (\Gamma (\infty))$ does not depend on $\alpha$.
\end{prop}

\begin{proof}Consider $\alpha' : 
\Upsilon \to \Gamma^{N'}$
another definable embedding of the form
$(\val (f'_1), \cdots, \val(f'_{N'}))$ with the functions 
$\val (f'_i)$ 
$(w,+)$-generating $\SS(\Upsilon)$ up to constant functions.
After adding zeroes we may assume $N = N'$.
Since  the functions 
$\val (f_i)$ are 
$(w,+)$-generating $\SS(\Upsilon)$ up to constant functions,
there exists a finite partition of $\Upsilon$ into
definable pieces $\Upsilon_j$ such that on each $\Upsilon_j$ 
we may write
$(\val (f'_i)) = M_j ((\val (f_i))) + a_j$ with $M_j$ a matrix with coefficients in
$\mathbb{Z}$ and $a_j \in \Gamma^N$.
Exchanging $\alpha$ and $\alpha'$ we get that the matrix $M_j$ lies  in $\mathrm{GL}_N(\mathbb{Z})$.
\end{proof}

Thus, to any  $\Gamma$-internal subset  $\Upsilon$ 
 of $X^{\#}_{\mathrm{gen}}\subseteq \widehat X$,
 we may assign a well defined motivic volume
$\mathrm{MV} (\Upsilon)$ in the ring
$K_+ (\Gamma (\infty))$, namely 
 the class of 
$\alpha (\Upsilon)$  for any  embedding $\alpha$ as above.

If $\Upsilon$ is contained in a definably compact set, $\alpha (\Upsilon)$ is bounded, thus
$\mathrm{MV} (\Upsilon)$ lies in
 $F^n K^{\mathrm{bdd}} _+ (\Gamma (\infty))$ and we can consider
 its $n$-dimensional volume
 $\mathrm{vol}_{n, \beta} (\mathrm{MV} (\Upsilon))$ in $R$
  for any embedding $\beta: \Gamma \to R$ with $R$ a real closed field.
  Similarly, any definable subset of $\Upsilon$ of o-minimal dimension $m \leq n$ contained in 
 a definably compact set has an $m$-dimensional volume in $R$.

\subsection{Berkovich variants}
These constructions admit direct variants in the Berkovich setting which are transfered
from the  previous section \ref{mvs} similarly as in the proof of 
Theorem \ref{mtberk}.

Fix an algebraically closed non-archimedean complete field $F$ with value group $A$.
Let $X$ be an integral $F$-scheme of finite type and of  dimension $n$.
Let $S\subseteq X\an$ be an algebraic skeleton as in the statement of Theorem \ref{mtberk}.
Then one can assign similarly as above a well defined class
$\mathrm{MV} (S)$ to $S$ in 
 in
$K_{+} (\mathbb{R}(\infty))$. Furthermore, 
if $S$ is relatively compact, since $A \subseteq \mathbb{R}$, one can consider its 
$n$-dimensional volume
 $\mathrm{vol}_n (\mathrm{MV} (S))$ in $\mathbb{R}$, or more generally its $m$-dimensional volume if
 $S$ of dimension $\leq m$.

\begin{rema}Note that all the invariants defined above (motivic and actual volumes) are invariant 
under birational automorphisms and Galois actions.
\end{rema}

\appendix

\section{Abhyankar valuations are defectless: a model-theoretic proof}\label{app:abh}

Let~$K$ be a field equipped with a Krull valuation $v$ and let~$L$ be a finite extension of~$K$. Let~$v _1,\ldots, v_n$
be the valuations on~$L$ extending~$v$, and for every~$i$, let~$e_i$ and~$f_i$ be the ramification and inertia indexes 
of the valued field extension $(K,v)\hookrightarrow (L,v_i)$. One always has
$\sum e_if_i\leqslant [L:K],$ and the extension~$L$ of the valued field~$(K,v)$ is said to be {\em defectless}
if equality holds. We shall say that~$(K,v)$ is {\em defectless}
if every finite extension of it is defectless (such a field is also often called \textit{stable} in the literature, 
but we think that defectless is a better terminology, if only because stable has a totally different meaning in model theory). 

We shall use here the notion of the \textit{graded}
residue field of a valued field in the sense of Temkin, 
see 
\cite{temkin2004} (we will freely apply the basic facts about
graded commutative algebra which are proved therein). A more model-theoretic approach of the latter was 
introduced independently by the second
author and Kazhdan in \cite{HK} 
with the notation $\mathrm {RV(\cdot)}$ which we have decided
to adopt here. 
The key point making this notion
relevant for the study of defect
is the following obvious remark: the product~$e_if_i$ can also be interpreted as the degree of
the graded
residue extension 
$\mathrm{RV}(K,v)\hookrightarrow \mathrm{RV}(L,v_i)$.

{\em Examples.} Any algebraically closed valued field is defectless; any complete discretely valued field is
defectless ; 
the function field of an irreducible normal algebraic variety, endowed with the discrete valuation associated to 
an irreducible divisor, is defectless; any valued field
whose residue characteristic is zero is defectless.

The purpose of this appendix is to give a new proof of the following well-known theorem.

\begin{theo}\label{theo-abhyankar-defectless}
Let $(K,v)$ be a defectless valued field, and let $G$ be an abelian ordered group containing
$v(K^\times)$.
Let $g=(g_1,\ldots, g_n)$ be a finite family of elements of $G$. Endow 
$K(T)=K(T_1,\ldots,T_n)$
with the ``Gauss extension $v_g$ of $v$ with parameter $g$",
\ie
\[v_g(\sum a_I T ^I)=\min_I v(a_I)+Ig.\]
The valued field $(K(T), v_g)$ is still defectless. 
\end{theo}

\medskip
This result has been given several proofs by Gruson, Temkin, Ohm, Kuhlmann, Teissier (see \cite{gruson}, \cite{temkin2010}, \cite{ohm1989}, \cite{kuhl}, \cite{teissier}). To our knowledge, the
first proof working in full generality was that of Kuhlmann, the preceeding proofs requiring some additional assumptions on~$K$ and/or on the~$g_i$. 
Our proof follows a more model theoretic route, relying on  the definability of the 
defectless locus.

\begin{proof}
It is rather long. 
Before writing it down, let us describe roughly its main steps. 
One first reduces to the case where $n=1$ by arguing inductively
(and one sets $T=T_1$ and $g=g_1$) and then to the case
where $K$ is algebraically closed (\ref{reduc-case-algclosed}), which
requires to understand what happens
when one performs a finite ground field extension of $K$, and this is the point where defectlessness of $K$ is needed. 

Then one shows that if $(L,w)$ is an algebraically closed valued extension of $K$ whose value group 
contains $\val(K^\times)+\Z g$, then $F$ is defectless over $(K(T),v_g)$ if and only if $F_L$
is defectless over $(L(T), w_g)$ (\ref{subsubsection-rational-radius}). This ultimately relies on the description
of definable maps from $\Gamma$ to the space of lattices (or semi-norms) on a vector space
(\cite{HL}, Lemma 6.2.2), which itself rests on the 
work \cite{haskell-h-m2006}
on 
imaginaries in $\acvf$. This enables us to assume that the valuation of $K$ is non-trivial and $g\in v(K^\times)$. 
Now one proceeds as follows: 

\begin{itemize}
\item[(A)] One shows (\ref{defectless-definable})
that there exists a $K$-definable subset $D$ of $\Gamma$ so that for every $h\in v(K^\times)$ 
the extension $F$ of $(K(T),v_h)$ is defectless if and only if $h\in D(K)$ (and this holds universally, \ie~
this equivalence remains true after base change from $K$ to an arbitrary model of $\acvf$); 

\item[(B)] One shows that $D$ is both definably open and definably closed
(\ref{d-openandclosed}) and non-empty (\ref{d-nonempty}),
so that $D$
is the whole of $\Gamma$; in particular $g\in D$, which ends the proof.
\end{itemize}

Statement (A) rests on the fact that on a smooth projective curve there exists a line bundle whose quotients 
of non-zero global sections generate (universally) the group of invertible rational functions (this follows from the
Riemann-Roch theorem); the proof uses this fact both directly and indirectly, through one 
of its important consequences in Hrushovski-Loeser's theory: 
definability (and not merely pro-definability, as in higher dimensions) of the stable completion of a curve.
And statement (B) ultimately relies on defectlessness of the function field of a curve equipped with the discrete valuation associated
to a closed point.

\subsubsection{First easy reduction}
By a straightforward induction argument, we reduce to the case where $n=1$, and we write now
$T$ instead of $T_1$ and $g$ instead of $g_1$. 

\subsubsection{Reduction to the case where $K$ is algebraically closed}
\label{reduc-case-algclosed}
We choose an arbitrary extension $w$
of $v$
to an algebraic closure $\overline K$ of $K$, and we endow
the field $\overline K(T)$
with the Gauss valuation $w_g$. 
We assume that 
$(\overline K(T),w_g)$ is defectless, and we want to prove that $(K(T),v_g)$ is
defectless too; this is the step in which our defectlessness assumption
on $K$ will be used. So, let $F$ be a finite extension of $K(T)$, and let us prove that it is defectless. 

We begin with a general remark which we will
use several times.  Let $K'$ be a finite extension of $K$. For every extension $v'$ of $v$
on $K'$ there is a unique extension of $v_g$ on $K'(T)$ whose restriction to $K'$ coincides with $v'$, namely the Gauss valuation
$v'_g$ (indeed, for such an extension $\mathrm{RV}(T)$ will be transcendental over $\mathrm{RV}(K')$, so this extension
is necessarily a Gauss extension of $v'$). Then it follows by a direct explicit computation that 
\[\mathrm{RV}(K'(T))=\mathrm{RV}(K(T))\otimes_{\mathrm{RV}(K)}\mathrm{RV}(K'),\]
(where
$K'$ is endowed with $v'$ and $K'(T)$ with
$v'_g$)
which implies that $K'(T)$ is a defectless extension of $K(T)$.

Let us first handle the case where $F$ is separable over $K(T)$. Let
$K'$ be the separable closure of $K$ in $F$. By the remark above, 
$K'(T)$ is a defectless extension of $K(T)$, and it is therefore
sufficient to prove that $F$ is a defectless extension of
$K'(T)$, thus we can assume that $K'=K$. 
The tensor product $L:=\overline K\otimes_K F$ is then a field, and $L$ is a defectless extension of $\overline K(T)$ since $\overline 
K(T)$ is defectless by assumption. Let $w_1,\ldots, w_d$ be the extensions of $w_g$ to $L$; for every $i$, let
$L_i$ be the valued field $(L,w_i)$. We have by assumption
\[[F:K(T)]=[L:\overline K(T)]=\sum_i  \, [\mathrm{RV}(L_i):\mathrm{RV}(\overline K(T))].\]
Now each $\mathrm{RV}(L_i)$ is a finite extension of $\mathrm{RV}(\overline K(T))$, so it is defined
over $\mathrm{RV}(E(T))$ for $E$ a suitable finite extension of $K$ contained 
in $\overline K$, which can be chosen to
work for all $i$.  Let us 
set
\[E_i=\mathrm{RV}(F\otimes_K E, w_i|_{F\otimes_K E}).\]
By construction, $E_i$
ontains
a graded subfield of degree $[\mathrm{RV}(L_i):\mathrm{RV}(\overline K(T))]$
over $\mathrm{RV}(E(T))$, so that
we have
\[[F\otimes_K E:E(T)]=[F:K(T)]=\sum_i \, [\mathrm{RV}(L_i):\mathrm{RV}(\overline K(T))]
\leq \sum_i \, [E_i:\mathrm{RV}(E(T))].\]
Then 
\[[F\otimes_K E:E(T)]=\sum_i \, [E_i:\mathrm{RV}(E(T))]\]
and $F\otimes_K E$ is a defectless extension of $E(T)$. Moreover, 
$E(T)$ is a defectless extension of $K(T)$ by the remark at the beginning of the proof. 
Therefore $F\otimes_K E$ is a defectless extension of $K(T)$ as well, which in turn forces
$F$ to be defectless over $K(T)$. We thus are done when $F$ is separable
over $K(T)$. 

Now let us handle the general case. 
In order to prove that $F$ is defectless over $K(T)$ we may enlarge $F$, and so 
we can assume that it is normal over $K(T)$.  Let $F_0$ 
be the subfield of $F$ consisting of Galois-invariant elements. This is a purely inseparable
extension of $K(T)$, and $F$ is separable (and even Galois) over $F_0$. 
Since $F_0$ is a finite extension of $K(T)$, it is contained in $K_0(T^{1/p^m})$
for some integer $m$ and some purely inseparable finite extension $K_0$ of $K$ (indeed,
if $f\in K(T)$ then for every ~$\ell$ the $p^\ell$-th root $f^{1/p^{\ell}}$ is
contained in the radicial extension
generated by $T^{1/p^{\ell}}$ and the $p^{\ell}$-th roots
of the coefficients of $f$). 

It is now sufficient to prove that $F\otimes_{F_0}K_0(T^{1/p^m})$ (which is a field since
$F$ and $K_0(T^{1/p^m})$ are respectively separable
and purely inseparable over $F_0$) is defectless  over~$K(T)$. But $ F\otimes_{F_0}K_0(T^{1/p^m})$ is
separable over $K_0(T^{1/p^m})$, so it is 
defectless over $K_0(T^{1/p^m})$ by the above; and $K_0(T^{1/p^m})$ is defectless over 
$K(T)$ by direct computation, resting on the fact that $K_0$ is defectless over $K$, which ends this first step. 

\textit{We thus may and do assume from now on that $K$ is algebraically closed}.

\subsubsection{Reduction to the case of
a rational radius}
\label{subsubsection-rational-radius}
Let $F$ be a finite extension of $K(T)$,
and let 
$C$ be the normal projective $K$-curve
with function field $F$, equipped with 
the finite map $C\to \P^1_K$ inducing $K(T)\hookrightarrow F$. 
We want to prove that~$F$
is defectless over the valued
field $(K(T),v_g)$ and our purpose now is
to reduce to the case where $g$ belongs to $v(K^\times)$. 

Let us fix a non-trivially valued, algebraically closed extension $L$ of~$K$ whose
value group contains
$v(K^\times)+\Z g$; let $v_L$ denote the valuation of $L$.
We are going to prove that $F_L:=F\otimes_{K(T)}L(T)$
is defectless over $(L(T),v_{L,g})$ 
if and only if $F$ is defectless over $K(T)$, which will allow
 to replace $(K,v)$ with $(L,v_L)$ and thus assume that
$K$ is non-trivially valued (in other words, $K$
is a model of $\acvf$) and $g\in v(K^\times)$.

Let $w$ be any extension of~$v_{L,g}$ to~$F_L$; 
in what follows, $F_L$ and its subfields are understood as endowed with (the restriction of) $w$. 
The valuation $w$ on $F_L$ defines a type on $C_L$ over $L$, whose image on $\P^1_L$ is by design
the generic type on the closed ball of valuative radius $g$ (centered at the origin). This type is thus strongly
stably dominated and definable over $K\cup\{g\}$, see \cite{HL}, Proposition 8.1.2. 

Let~$E$ be a finite dimensional~$K$-vector subspace of~$F$.
It follows from the above that the restriction of~$w$ to~$L\otimes_K E$ is a norm which is definable with parameters
in~$K\cup\{g\}$, once a $K$-basis of $E$ is chosen. 
Otherwise said, identifying a norm on $E$ with its unit ball, there exists a $K$-definable function $\Phi$ from 
$\Gamma$ to the set of lattices of $E$ such that $w|_{L\otimes_KE}=\Phi(g)$. In view of the general description of such a 
$\Phi$
provided by \cite{HL}, Lemma 6.2.2,
this implies the existence of a basis~$e_1,\ldots, e_d$ of~$E$ over~$K$
and elements~$h_1,\ldots, h_d$ of 
$v(K^\times)\oplus \mathbb{Q} g$ such
that
\[(1)\;\;\;w\Bigl(\sum a_i e_i\Bigr)=\min v(a_i)+h_i\] for every $d$-uple
$(a_i)\in L^d$. 
Note that one thus has
\[(2)\;\;\;w(x)=\max_{x=\sum a_i\otimes y_i}\min_i (v(a_i)+w(y_i))\]
for all $x\in L\otimes_K E$.

It immediately follows 
from (1) that the \emph{graded}
reduction $\mathrm{RV}(L\otimes_K E)$ is
equal to~
$\mathrm{RV}(L)\otimes_{\mathrm{RV}(K)}
\mathrm{RV}(E)$.
A limit argument then 
shows that
$\mathrm{RV}(F_L)$ is 
nothing but the graded fraction field
of
$\mathrm{RV}(L)\otimes_{\mathrm{RV}(K)}\mathrm{RV}(F)$. As
$\mathrm{RV}(L(T))$ is
itself equal by a direct computation to the graded fraction field
of
the graded domain $\mathrm{RV}(L)\otimes_{
\mathrm{RV}(K)}\mathrm{RV}(K(T))$, 
we eventually get
\[\mathrm{RV}(F_L)=\mathrm{RV}(L(T))\otimes_{
\mathrm{RV}(K(T))}\mathrm{RV}(F).\]
In particular we have the equality
\[(3)\;\;\;[\mathrm{RV}
(F_L):\mathrm{RV}(L(T))]=[\mathrm{RV}(F):\mathrm{RV}(K(T)].\]

This holds for all extensions $w$ of $v_{L,g}$ to $F_L$ (we remind that 
$w$ is implicitly involved in the above equality). Let $\mathcal P$, \resp $\mathcal P_L$, be the set of extensions
of $v_g$ to $F$, \resp of $v_{L,g}$, to $F_L$. There is a natural restriction map
from $\mathcal P_L$ to $\mathcal P$, which is injective since 
formula (2) above ensures that any $w\in \mathcal P_L$
is uniquely determined by its restriction to $F$. 
We claim that this map is surjective as well. Indeed,
to see this, we may enlarge $F$ and assume it is Galois over $K(T)$. 
Now let $\omega\in \mathcal P$ and let $w$ be an arbitrary element of $\mathcal P_L$. 
The restriction $w|_F$ belongs to $\mathcal P$, so is equal to $\omega\circ \phi$
for some $\phi\in \mathrm{Gal}(F/K(T))$. Then $w\circ \phi^{-1}$ is a preimage
of $\omega$ in $\mathcal P_L$.

Therefore $\mathcal P_L\to \mathcal P$ is bijective. In view of $(3)$ above, this implies that
$F$ is a defectless
extension of~$(K(T), v_g)$ if and only if~$F_L$ is a defectless extension of~$(L(T),v_{L,g})$, 
as announced. 

\textit{Hence we may and do
assume from
now on that~$g\in v(K^\times)$ and that $K$ is a model of $\acvf$}.

\subsubsection{Some specialisations}
\label{subsubsection-specialisation}Let~$h\in v(K^\times)$.  Let us choose~$\lambda\in K$ such that~$v(\lambda)=h$ and let~$\tau$ be the image of~$T/\lambda$ in the residue field $k$ of~$(K(T),v_h)$; note that~$k=\mathsf{res}(K)(\tau)$, and that 
$\tau$ is transcendental over $\mathsf{res}(K)$.
 Let~$h^- $ and~$h^+$ be elements of an abelian ordered group containing~$v(K^\times)$ which are infinitely close to~$h$ (with respect to~$v(K^\times)$), with 
$h^-<h<h^+$. The valuation $v_{h^-}$, \resp
$v_{h^+}$ is the composition of $v_h$ and of the discrete valuation $u_\infty$, resp.~$u_0$, of $k$ that corresponds to~$\tau=\infty$, \resp $\tau=0$, and the extensions of $v_{h^-}$, \resp $v_{h^+}$, to~$F$ are compositions of extensions of~$v_h$ and of extensions of
$u_\infty$, \resp $u_0$. Since~$(k, u_0)$ and $(k, u_\infty)$ are defectless,
we see that the following are equivalent : 

\begin{enumerate}[i]
\item $F$ is a defectless extension of~$(K(T), v_{h^-})$ ;

\item $F$ is a defectless extension of~$(K(T),v_h)$ ; 

\item $F$ is a defectless extension of~$(K(T),v_{h^+})$.  
\end{enumerate}

In the same spirit, let~$\theta$ be
an element of an
abelian ordered group containing~$v(K^\times)$ and larger than any element of 
$v(K^\times)$. The valuation~$v_\theta$ is the composition of the discrete valuation $\omega$ of~$K(T)$ corresponding to the closed point~$T=0$ and of the valuation of~$K$. Since both~$(K,v)$ and
$(K(T),\omega)$ are defectless, $(K(T), v_\theta)$ is defectless; in particular, $F$ is a defectless extension of~$(K(T),v_\theta)$.

\subsubsection{Definability of the defectless locus}
\label{defectless-definable}
Our purpose is now to prove the existence of a 
$K$-definable subset $D\subset \Gamma$ such that for
every model $(L,w)$ of $\acvf$ containing $K$ and every $h\in 
w(L^\times)$, the extension $F_L$ of $(L(T),w_h)$ is
defectless if and only if $h\in D(L)$. We first note that in view
of 
\ref{subsubsection-specialisation}, $F_L$ is a defectless
extension of $(L(T),w_h)$ if and only if it is a defectless 
extension of $(L(T), w_h^{+})$, and it is the latter property we shall
focus on. 

Let~$X$ be an irreducible, smooth, projective curve over 
$K$ whose function field is isomorphic to $F$, and such that~$K(T)\hookrightarrow F$ is induced by  a finite map~$f\colon X\to \P^1_K$; the latter induces
a map $\widehat  f\colon \widehat X\to \widehat{\P^1_K}$.
It
follows from Riemann-Roch that there exists a line bundle $\mathscr L$ on~$X$ such 
that the quotients $s/t$ for $s$ and $t$ running through 
the set of non-zero global sections
of $\mathscr L$ generates $K(X)^\times$ universally
(see \cite{HL}, 7.1; this is the key input for the proof therein that $\widehat X$
is definable, and not merely pro-definable). 

We identify $\Gamma$ with the standard skeleton $\Sigma_1
\subset \widehat{\P^1_K}$; let $\Delta$ be its pre-image in 
$\widehat X$. The set $\Delta$ is $K$-definable and $\Gamma$-internal
(this follows directly from the definability of $\widehat X$ and $\widehat{\P^1_K}$
and the fact that $\widehat X\to \widehat{\P^1_K}$ has finite fibers,
with no need to invoke Theorem \ref{theo-union-skeleta}).
There exists a finite $K$-definable set $S\subset \Delta$ such that 
$\Delta\setminus S$ is a disjoint union $\coprod
_{I\in \mathscr I} I$ of definably open intervals, each of which maps
homeomorphically onto a definable open interval in 
$\Gamma$ (and is equipped with the orientation and the metric
inherited from
$\Gamma$). 

For every $\omega\in \Delta$, we denote by $\mathscr I(\omega)$ the
subset of $\mathscr I$ consisting of those intervals $I$ such that $\omega\in I$
or $\omega$ is the left endpoint of $I$. 
For every $I\in \mathscr I(\omega)$, we denote by $s(I,\omega)$ the
set of all
possible slopes of $\val(s/t)$ for $s$ and $t$ non-zero global sections of
$\mathscr L$ along
the germ of branch emanating rightward from $\omega$
and induced by $I$.
By finite-dimensionality of
$\H^0(X,\mathscr L)$ all sets $s(I,\omega)$
are finite
and the asssignment $\omega \mapsto (\mathscr I(\omega), 
(s(I,\omega))_{I\in \mathscr I(\omega)})$ is $K$-definable. 

Let $(L,w)$
be a model of $\acvf$ containing $K$, 
let $\omega\in \Delta(L)$ and let $I\in \mathscr I(\omega)$. 
The germ of branch emanating rightward from $\omega$
and induced by $I$ defines a valuation $v(I,\omega)$ refining 
$\omega$. The image of $\omega$ in $\widehat{\P^1}(L)$ 
is equal to $w_h$ for some
$h\in w(L^\times)$; thus $v(I,\omega)$ lies above $w_{h^+}$. 
The ramification index $e(I,\omega)$
of 
$v(I,\omega)$ over $w_{h^+}$ is the greatest $N>0$ such that there exists a 
non-zero $L$-rational function on $X$ whose valuation has slope $1/N$ 
along the germ of branch emanating rightward from $\omega$
and induced by $I$. But since the group of non-zero rational functions on
$X$ is universally generated by quotients of non-zero global sections of $\mathscr L$, this integer $e(I,\omega)$
can be read off from the finite set of slopes $s(I,\omega)$ (it is nothing but the lcm of their denominators). 

Now $F_L$ is a defectless extension of $(L(T), w_{h^+})$ if and only if the sum
of all the ramification indexes of $v(I,\omega)$ for $\omega$ above $w_h$
and $I\in \mathscr I(\omega)$ is equal to $[F:K(T)]$.
Thus whether $F_L$ is a defectless extension of $(L(T),w_{h^+})$ or not 
can be read off from the sets of slopes $s(I,\omega)$ for $\omega$
above $w_h$ and $I\in \mathscr I(\omega)$; the existence of the required
$K$-definable set $D$ follows immediately.

\subsubsection{Conclusion}

Our purpose is to prove that $F$ is a defectless extension of
$(K(T),v_g)$ or, in other words, that $g\in D(K)$, and we are in fact
going to prove that $D$ is the whole of $\Gamma$. For this, it suffices to show
that $D$ is both definably open and definably closed and non-empty.

\paragraph{The set $D$ is both definably open and definably closed}\label{d-openandclosed}
Let $h\in v(K^\times)$, and let
$(L,w)$ be a model of $\acvf$ containing $K$ and
such that $w(L^\times)$ contains two elements $h^+$
and $h^-$ infinitely close to $h$ with respect to $v(K^\times)$
and with $h^-<h<h^+$. In view of \ref{subsubsection-specialisation}, 
$F$ is a defectless extension of $(K(T),v_h)$ if and only if 
it is a defectless extension of $(K(T),v_{h^+})$, if and only if
it is a defectless extension of $(K(T),v_{h^-})$. Using
\ref{subsubsection-rational-radius}, this implies that 
$F$ is a defectless extension of $(K(T),v_h)$ if and only if 
$F_L$ is a defectless extension of $(L(T),w_{h^+})$, if and only if
$F_L$ is a defectless extension of $(L(T),w_{h^-})$. Hence
if $h\in D(K)$ then $h^+$ and $h^{-}$ belong to
$D(L)$, and if $h$ belongs to $(\Gamma\setminus D)(K)$, then 
$h^{-}$ and $h^{+}$ belong to  $(\Gamma\setminus D)(L)$. This shows
that both $D$ and its complement in $\Gamma$ are definably open, hence
$D$ is both definably open and definably closed. 

\paragraph{The set $D$ is non-empty}\label{d-nonempty}
Now let $(L,w)$ be a model of $\acvf$ containing $K$ such that $w(L^\times)$ contains
an element $\theta$ larger that any element of $v(K^\times)$. 
We have seen in \ref{subsubsection-specialisation}
that $F$ is a defectless extension of $(K(T),v_\theta)$. Thus
by \ref{subsubsection-rational-radius} $F_L$ is a defectless extension of $(L(T),w_\theta)$. Hence $\theta\in D(L)$ and $D$ is non-emtpy.
\end{proof}

\bibliographystyle{smfalpha}
\bibliography{paper}

\end{document}